\documentclass[11pt]{article}
\usepackage{a4, amsmath,amssymb, graphicx, hyperref}
\usepackage{enumerate,amsthm, fancyhdr, multicol, color}
\newtheorem{theorem}{Theorem}
\newtheorem{defn}{Definition}

\newtheorem{lemma}{Lemma}
\newtheorem{proposition}{Proposition}
\newtheorem{corollary}{Corollary}
\newtheorem{remark}{Remark}

 \newcommand{\N}{\mathbb{N}}

 \newcommand{\R}{\mathbb{R}}

 \newcommand{\E}{\mathbb{E}}
  \newcommand{\PP}{\mathbb{P}}

  \newcommand{\eq}{\begin{equation}}
  \newcommand{\qe}{\end{equation}}
 
 \def\qed{\hfill \mbox{\rule{0.5em}{0.5em}}}

 \newcommand{\bea}{\begin{eqnarray}}
\newcommand{\ena}{\end{eqnarray}}
\newcommand{\beas}{\begin{eqnarray*}}
\newcommand{\enas}{\end{eqnarray*}} 
\numberwithin{equation}{section}

\begin{document}

\title{A new approach to the Stein-Tikhomirov method \\
\small with applications to the second Wiener chaos and Dickman convergence} \date{}

\author{Benjamin Arras, Guillaume Mijoule, Guillaume Poly and Yvik Swan}
\maketitle

 \begin{abstract}
   In this paper, we propose a general means of estimating the rate at
   which convergences in law occur. Our approach, which is an
   extension of the classical Stein-Tikhomirov method, rests on a new
   pair of linear operators acting on characteristic functions. In
   principle, this method is admissible for any approximating sequence
   and any target, although obviously the conjunction of several
   favorable factors is necessary in order for the resulting bounds to
   be of interest. As we briefly discuss, our approach is particularly
   promising whenever some version of Stein's method applies. We apply
   our approach to two examples. The first application concerns
   convergence in law towards targets $F_\infty$ which belong to the
   second Wiener chaos (i.e.\ $F_{\infty}$ is a linear combination of
   independent centered chi-squared rvs). We detail an application to
   $U$-statistics.  The second application concerns convergence
   towards  targets belonging to the generalized Dickman family of
   distributions. We detail an application to a theorem from number
   theory. In both cases our method produces bounds of the correct
   order (up to a logarithmic loss) in terms of quantities which occur
   naturally in Stein's method.
 \end{abstract}
 
 
\tableofcontents
\section{Introduction}



Let the sequence $(F_n)_{n\ge 1}$ of random variables have
characteristic functions (CF)
$\phi_n(\cdot) = E \left[ e^{i \cdot F_n} \right]$ and let
$F_{\infty}$ have CF
$\phi_{\infty}(\cdot) = E \left[ e^{i \cdot F_\infty}
\right]$. Suppose that $(F_n)_{n\ge 1}$, hereafter referred to as the
\emph{approximating sequence}, converges in distribution to
$F_{\infty}$, hereafter referred to as the \emph{target}. An important
question is to estimate precisely the \emph{rate} at which this
convergence occurs.  In this paper we provide a new approach to this
classical problem, which is based on a version of Stein's method for
characteristic functions known as the Stein-Tikhomirov method. Before
going into the details of our proposal we provide a brief overview of
the context in which it is inscribed.

\subsection{Four classical methods}
\label{sec:three-class-meth}
Among the many available methods for deriving rates of convergence in
probabilistic approximation problems, arguably the most classical one
consists in seeking precise pointwise control of
\begin{equation*}
  \Delta_n(t) =   \phi_n(t) - \phi_\infty(t) 
\end{equation*}
(or, somewhat equivalently,
$ \left| \tilde{\Delta}_n(t) \right|= \left| {\phi_n(t)}/{\phi_\infty(t)}-1
\right|$).  There is a double justification for the study of this
quantity. First, Levy's continuity theorem guarantees that proximity
of $|\Delta_n(t)|$ with 0 
necessarily captures proximity in distribution between $F_n$ and
$F_{\infty}$ in some appropriate sense. Second, the fact that in many
cases the limiting distributions are infinitely divisible ensures (via
the L\'evy-Khintchine representation) that the target's CF has a
manageable form. Of course, once $\Delta_n(t)$ is estimated it remains
to translate the resulting bounds into a bound on some meaningful
probabilistic discrepancy. A fundamental result in this regard is
\emph{Berry and/or Esseen's inequality} which states that, under
classical regularity conditions and provided the target has density
bounded by $m$, for all $x$ and all $T >0$,
  \begin{equation}
    \label{eq:1}
    \left| \mathbb{P}(F_n\le x) - \mathbb{P}(F_{\infty}\le x) \right| \le \frac{1}{\pi}
    \int_{-T}^T \left| \frac{\phi_n(t) - \phi_\infty(t)}{t} \right|dt
    +  \frac{24m}{\pi T},
  \end{equation}
  see e.g.\ \cite[Theorem 1.5.2 p 27]{IL71} for a complete statement.
  Equation \eqref{eq:1} provides bounds on the \emph{Kolmogorov
    metric} between $F_n$ and $F_{\infty}$
\begin{equation}
  \label{eq:28}
  \mathrm{Kol}(F_n, F_{\infty}) = \sup_{x \in \R} 
  \left| \mathbb{P}(F_n \le x) - \mathbb{P}(F_{\infty}\le x) \right|
\end{equation}
as soon as the estimates on $\left| \Delta_n(t) \right|/t$ are sufficiently
integrable. This approach was pioneered in the classical Berry-Esseen
theorems \cite{E42,B41} and is also at the heart of the
Stein-Tikhomirov method about which more shall be said below.

An alternative method which is relevant to the present paper and which
actually predates Berry and Esseen's CF approach is Lindeberg's
``replacement trick'' from \cite{L22} (which dates back to 1922, see
e.g.\ \cite{EL17} for an overview and up-to-date references). Here the
idea bypasses the Fourier approach entirely by controlling the
distance between $F_n$ and $F_{\infty}$ in terms of quantities
\begin{equation*}
  \Delta_n(h) = \mathbb{E} \left[  h(F_n)
    \right]  - \mathbb{E}\left[ h(F_{\infty}) \right]
\end{equation*}
with $h$ some aptly chosen \emph{test function}. Note how Lindeberg's
$\Delta_n(h)$ gives Berry and Esseen's $\Delta_n(t)$ by taking
$h(x) = e^{i t x}$. Then if the approximating sequence is a sum and
the target is Gaussian, Lindeberg's method consists in exploiting the
infinite divisibility of the Gaussian law to replace the summands by
their Gaussian counterparts and control the resulting expressions via
probabilistic Taylor expansions. This approach requires to work with
test functions which are smooth with a certain number of bounded
derivatives and leads to bounds in the \emph{smooth Wasserstein
  metric}
\begin{equation}\label{eq:6}
  \mathcal W_p (F_n,F_{\infty})= \sup_{h \in \mathcal{H}_{p} } \left|
    \mathbb{E} \left[ h(F_n) \right] - \mathbb{E} \left[ h(F_{\infty}) \right]\right|
\end{equation}
with $p\geq 1$ and
$\mathcal{H}_{p} = \{ h \in \mathcal C^p(\R) \; | \;
||h^{(k)}||_\infty \leq 1, \forall\; 0 \leq k \leq p \}.$ We stress
the fact that there is no direct subordination relationship between
the Kolmogorov distance \eqref{eq:28} and the smooth Wasserstein
distances \eqref{eq:6} (although in the case of bounded densities the
former is dominated by the \emph{square root} of the latter); rates of
convergence in either metrics are relevant in their own right.

Over the last few decades, an alternative and very general approach
has attracted much attention: Stein's method (SM).  Here the idea, due
to \cite{S72}, is to tackle the problem of estimating general
\emph{integral probability metrics}
\begin{equation}
  \label{eq:19}
    d_{\mathcal{H}} (F_n, F_{\infty})  = \sup_{h\in \mathcal{H}} \left|
    \mathbb{E} \left[ h(F_n) \right]  - \mathbb{E} \left[
      h(F_{\infty}) \right] \right|
\end{equation}
(note how both \eqref{eq:28} and \eqref{eq:6} are of this form) in a
spirit close to that underlying Lindeberg's trick, with one
supplementary sophistication: the \emph{Stein operator}. This is a
linear operator $\mathcal{A}_{\infty}$ acting on
$\mathcal{F}_{\infty}$ some large class of functions, with the
property that
$ \mathcal{L}(F) = \mathcal{L}(F_{\infty}) \Longleftrightarrow
\mathbb{E} \left[ \mathcal{A}_{\infty}(f)(F) \right] = 0$
$\forall f \in \mathcal{F}_{\infty}$.  Similarly as for the CF
approach, the characterizing nature of the operator guarantees that
\begin{equation*}
D_n(f) =  \mathbb{E} \left[ \mathcal{A}_{\infty}f(F_n) \right] 
\end{equation*}
ought to capture some aspect of the proximity in distribution between
the approximating sequence and the target.  There exist many target
distributions (both discrete and continuous) which admit a tractable
operator $\mathcal{A}_{\infty}$; the most classical examples are the
Gaussian $F_{\infty} \sim \mathcal{N}(0, 1)$ with operator
$\mathcal{A}_{\infty}f(x) = f'(x) - xf(x)$ and the Poisson
$F_{\infty} \sim \mathcal{P}(\lambda)$ with operator
$\mathcal{A}_{\infty}f(x) =\lambda f(x+1)- xf(x)$ (see
\cite{Stein1986} for an introduction; more recent references will be
pointed to in Section \ref{sec:operator-l_infty}).  The connection
between the quantity $D_n(f)$ and integral probability metrics goes as
follows: given a target $F_{\infty}$ along with its operator and class
$(\mathcal{A}_{\infty}, \mathcal{F}(F_{\infty}))$, if for all
$h\in \mathcal{H}$ the solutions to
$\mathcal{A}_{\infty}f_h(x) = h(x) - \mathbb{E} \left[ h(F_{\infty})
\right]$ belong to $\mathcal{F}(F_{\infty})$ then
\begin{equation}
  \label{eq:17}
  d_{\mathcal{H}} (F_n, F_{\infty}) = \sup_{h\in \mathcal{H}} \left| D_n(f_h) \right|.
\end{equation}
SM advocates to tackle the problem of estimating
$ d_{\mathcal{H}} (F_n, F_{\infty})$ by concentrating on the right
hand side of \eqref{eq:17}. Remarkably, if the operator is well chosen
then the above is a proficient starting point for many approximation
problems, particularly in settings where the traditional methods break
down (e.g.\ in the presence of dependence for the CLT). A crucial
ingredient of SM is that one needs to dispose of sharp estimates on
the solutions $f_h$ in terms of the test functions $h$; this can often
reveal itself a difficult problem, see e.g.\ \cite{RO12,Do14}.

The last method of interest to this paper is the
\emph{Stein-Tikhomirov} (ST) method, which is a combination of the SM
and the CF approaches.  Here the idea is to take (in the notations of
Stein's method)
$\mathcal{F}_{\infty} = \left\{ e^{i t \cdot}, \, t\in \R \right\}$
for which the Stein discrepancy $D_n(\cdot)$ becomes
\begin{equation*}
  D_n(t) =   \mathbb{E} \left[ \mathcal{A}_{\infty}(e^{i t \cdot}) (F_n)  \right].
\end{equation*}
If the target is standard Gaussian and the approximating sequence has
finite first moment then $\mathcal{A}_{\infty}f(x) = f'(x) -xf(x)$ and
\begin{align*}
  \left| D_n(t) \right| & =  \left|  \mathbb{E} \left[it e^{itF_n} - F_n\, e^{i t F_n} 
           \right]  \right| = \left| t\phi_n(t) + \phi_n'(t) \right|.
\end{align*}
Note how $t\phi_{\infty}(t) +\phi_{\infty}'(t) = 0$ at all $t$ so that
$D_{\infty}(t)=0$ at all $t$. Next take the approximating sequence of
the form $F_n = n^{-1/2} \sum_{i=1}^n \xi_i$, a standardized sum of
weakly dependent summands. Tikhomirov's tour de force from
\cite{tikhom1981} is to show that there exist constants $B$ and $C$
such that
\begin{equation}
  \label{eq:5}
\phi_n'(x) + x \phi_n(x) = \theta_1(x) \frac{x^2 B m}{\sqrt n}
  \phi_n(x) + \theta_2(x) \frac{x^2 C m}{n} =: R_n(x)
\end{equation}
for $|x| \le D_n \sqrt{n}$, with $\theta_j(\cdot), j=1, 2$ bounded
functions, and $m$ an integer determined by the nature of the
dependency of the summands (see \cite[Section 2, equation
(2.2)]{tikhom1981}) which may depend on $n$ and $D_n$ a strictly
positive sequence tending to $0$ when $n$ tends to
$+\infty$. Integrating 
equation (\ref{eq:5}) immediately leads to 
\begin{align*}
\phi_n(x)=\exp\bigg(-\frac{x^2}{2}+\theta(x)\frac{|x|^3Bm}{\sqrt{n}}\bigg)+\theta(t)\dfrac{B|t|m}{n}
\end{align*}
for all $|x| \le D_n \sqrt{n}$.  Plugging this expression into
(\ref{eq:1}) leads to Berry-Esseen type bounds of the correct order in
Kolmogorov metric for central limit theorem for weakly dependent
sequences. The crucial point is that the leading term appearing in the
difference between the characteristic functions (see for e.g. the
inequality p. 800 of \cite{tikhom1981}, just before Section $4$) has
bounded integral over $[-T, T]$ not depending on $T$ so that no loss
of order is incurred via the optimization process when transferring
the rate of convergence in CFs into a rate of convergence in
Kolmogorov via Esseen's identity.

\subsection{Purpose and outline of the paper}
\label{sec:purpose-paper}

To the best of our knowledge, Tikhomirov's approach has been
principally explored for Gaussian targets in the context of CLT's for
sums of dependent random variables. We refer the reader to the paper
\cite{RO17} for an application to random graphs and for recent
references. We are not aware of any application of the method for non
Gaussian targets, although several pointers are available in the
conference abstract and slides \cite{tikhom2015}.  The purpose of this
paper is to extend the applicability of ST method in two ways. First,
we provide an alternative to Berry's inequality \eqref{eq:1} which
allows to translate bounds on
$|\Delta_n(t)| = \left| \phi_n(t) - \phi_{\infty}(t) \right|$ into
bounds on the smooth Wasserstein metrics \eqref{eq:6}; these bounds
are applicable under much less stringent integrability assumptions on
$\Delta_n(t)$ and lead to rates of the correct order (up to a
logarithmic penalty).  See Section \ref{sec:bounds-char-funct} for
details. Second, we provide a new vision of the Stein-Tikhomirov
approach which highlights the robustness of the method towards a
change in target and approximating sequence and, in principle,
provides a method for obtaining bounds on $\Delta_n(t)$ for virtually
any pair of distributions.  Of course, a new method is only as
interesting as its applications and we illustrate our approach by
presenting results concerning two non trivial approximation problems
which are both outside the range of current literature (although the
field is very active and several new references have appeared since
the first version of this work was made available on arXiv, see
Remark~\ref{sec:purp-outl-paper}).

The first application, developed in Section
\ref{sec:bounds-targ-second}, concerns targets $F_\infty$ which belong
to the second Wiener chaos (i.e.\ $F_{\infty}$ is a linear combination
of independent centered chi-squared rvs); our method produces
non-trivial bounds in terms of quantities which occur naturally in
Nourdin and Peccati's version of Stein's method. The second
application, developed in Section \ref{sec:quant-bounds-gener}, is
towards a target belonging to the generalized Dickman family of
distributions. This intractable distribution has long been recognized
as an important limit distribution in number theory and analysis of
algorithms. Here again our method provides non trivial general bounds.
In both cases our bounds are the first of their kind in such general
problems and, whenever competitor results exist, provide competitive
estimates (up to a logarithmic loss). A specific and hopefully
exhaustive overview of the current relevant literature will be
provided in each section.  We stress the fact that our two
applications concern approximation problems which are of an entirely
different nature; this highlights the power of the combination of
Stein's and Tikhomirov's intuitions.

\begin{remark}
We have also applied our method to several classical
problem but, of course, have not been able in these cases to improve
on the known available bounds which are generally optimal. For the
interested reader we nevertheless include, in
Appendix~\ref{sec:stable-distributions}, the instantiation of our
approach towards computation of rates of convergence in the
generalized CLT for stable distributions. Although the rates we obtain
here are not as good as those obtained in the seminal references
\cite{Hall81} and \cite{JS05}, the method we develop nevertheless
remains of independent interest in part because it is the first
illustration of the theory of Stein operators developed for stable
targets. This being said we gladly acknowledge the precedence of the
-- so far unpublished -- work of Partha Dey and his collaborators, see
\cite{dey}.  
\end{remark}

\begin{remark}
  There are many references dealing with versions of the CF approach
  for convergence in distribution. An interesting method whose
  connection with ours would certainly be worth investigating is the
  so-called ``mod-$\phi$ convergence'' from \cite{FMN16, FMN17}.
\end{remark}

\begin{remark}\label{sec:purp-outl-paper}
  This article is a complete re-writing of a previous version
  submitted to arXiv in May 2016. During the period of revision, the
  preprint \cite{BG17} produced bounds directly comparable with the
  one we obtained in Section~\ref{sec:an-appl-towards}. More
  precisely, the authors of \cite{BG17} use a version of the classical
  approach to Stein's method specifically developed for Dickman
  targets in \cite{G17} (which is also posterior to the first version
  of the present paper) by which they obtain quantitative bounds for
  the $2$-Wasserstein smooth distance whereas we obtain bounds for the
  $3$-Wasserstein smooth distance of the same order up to some
  logarithmic loss (see inequality \eqref{rateDickW3} in
  Section~\ref{sec:an-appl-towards} and Theorem $1.1$ inequality (13)
  in \cite{BG17}). It is likely possible to improve our approach to
  obtain bounds in the same distance (i.e.\ in 2 Wasserstein smooth)
  for this example; the logarithmic loss is, however, an artefact of
  the generality of our method. We provide another example of Dickman
  quantitative convergence (see Theorem \ref{DickNumTh}) which does
  not seem to be a direct implication of the results contained in
  \cite{BG17}.
\end{remark}

\section{Bounds on characteristic functions and smooth Wasserstein}
\label{sec:bounds-char-funct}

The main result of this Section is the following transfer principle:
\begin{theorem}\label{thm:bounds-char-funct-1}
Let $X$, $Y$ be random variables with respective characteristic functions $\phi$ and $\psi$. We make the following hypothesis :
\begin{itemize}
\item[(H1)] \label{hyp:1} $\exists \, p \in \N$, $C'>0$ and $0<\epsilon<1$ such that
\eq
\forall \xi \in \R, \quad |\psi(\xi)-\phi(\xi)| \leq C' \epsilon |\xi|^p.
\qe
\item[(H2)] \label{hyp:2} $\exists \, \lambda>0, \exists C>0, \alpha>0 \forall A>0$, $\PP[|X|>A] \leq C e^{-\lambda A^\alpha}$ and $\PP[|Y|>A] \leq C e^{-\lambda A^\alpha}$.
\end{itemize}
Then,
\eq
\label{eq:thm1}
\mathcal{W}_{p+1} (X,Y) \leq 2\epsilon \left[ C+  \frac{2^{p-1}\sqrt{10}C'C_p}{\pi}\sqrt{1+\left(\lambda^{-1}\ln(1/ \epsilon)\right)^{\frac{1}{\alpha}}}\right],
\qe
where $C_{p}$ is a constant depending only on $p$. One can take $C_0=3$, $C_1=12$, $C_2=100$.

\end{theorem}
\begin{proof}
Let $f \in \mathcal C_c^\infty(\R)$, with support in $[-M-1,M+1]$, and $\hat f$ its Fourier transform. $f$ is in the Schwarz space $\mathcal S(\R)$, so that $f$ equals the inverse Fourier transform of $\hat f$. We have from Fubini's theorem that
\begin{align*}
\E[f(X) - f(Y)] &= \frac{1}{2\pi} \E\left[\int_\R (e^{iX\xi} - e^{iY\xi}) \hat f(\xi) d\xi\right]\\
&= \frac{1}{2\pi} \int_\R\E(e^{iX\xi} - e^{iY\xi}) \hat f(\xi) d\xi\\
&= \frac{1}{2\pi} \int_\R (\psi(\xi)-\phi(\xi)) \hat f(\xi) d\xi\\
&\leq \frac{C'\epsilon}{2\pi} \int_\R |\xi|^p |\hat f(\xi)| d\xi.
\end{align*}
We use the  Cauchy-Schwarz inequality and Plancherel's identity to get
\begin{align*}
 \int_\R |\xi|^p |\hat f(\xi)| d\xi &=  \int_\R |\xi|^p (1+|\xi|)|\hat f(\xi)|  \frac{1}{1+|\xi|}d\xi\\
 &\leq \sqrt{\int_\R |\xi|^{2p}(1+|\xi|)^2|\hat f(\xi)|^2 d\xi} \sqrt{\int_\R \frac{d\xi}{(1+|\xi|)^2}}\\
 &\leq 2 \sqrt{\int_\R |\xi|^{2p}(1+\xi^2)|\hat f(\xi)|^2 d\xi} \\
 &= 2 \sqrt{\int_\R |f^{(p+1)}(x)|^2 + |f^{(p)}(x)|^2 dx}.
\end{align*}
To sum up, if $f$ is a smooth function with support in $[-M-1,M+1]$,
then \eq \E[f(X) - f(Y)] \leq \frac{C'\epsilon\sqrt{2(M+1)}}{\pi}
\sqrt{||f^{(p+1)}||_\infty^2+||f^{(p)}||_\infty^2}.  \qe Consider now
the function $g(x) = a\int_0^x \exp\left(-\frac{1}{t(1-t)}\right) dt$
on $[0,1]$, $g(x)= 0$ for all $x\leq 0$ and $g(x) = 1$ for $x\geq 1$,
where $a>0$ is chosen so that $g \in \mathcal C^\infty(\R)$. Let
$K_n = \sup_{0\leq k \leq n} ||g^{(k)}||_\infty$; one can prove that
$K_1 \leq 3$, $K_2\leq 12$, $K_3\leq 100$.

Define the even function $\chi_M\in \mathcal C_c^\infty(\R)$ by $\chi_M(x) = 1$ on $[0,M]$ and $\chi_M(x) = g(M+1-x)$ if $x\geq M$. Let $F \in \mathcal C_c^\infty(\R)\cap  B_{0,p+1}$. Then
\begin{align*}
\E[F(X) - F(Y)] &\leq \E[F\chi_M(X) - F\chi_M(Y)] + \E[F(1-\chi_M)(X) + F(1-\chi_M)(Y)]\\
&\leq \frac{C'\epsilon\sqrt{2(M+1)}}{\pi} \sqrt{||(F\chi_M)^{(p+1)}||_\infty^2+||(F\chi_M)^{(p)}||_\infty^2} + 2 \PP[X \geq M]\\
&
\end{align*}
However, 
\begin{align*}
||(F\chi_M)^{(p)}||_\infty &\leq \sum_{k=0}^p \left(\begin{array}{c}p \\k\end{array}\right) ||F^{(p-k)}||_\infty ||\chi_M^{(k)}||_\infty\\
&\leq 2^p K_p.
\end{align*}
We obtain
\begin{align*}
\E[F(X) - F(Y)] &\leq \frac{C'\epsilon\sqrt{2(M+1)}}{\pi} \sqrt{2^{2p+2}K_{p+1}^2+2^{2p}K_p^2} + 2 \PP[X \geq M]\\
&\leq C'\epsilon \frac{2^{p}\sqrt{10(M+1)}K_{p+1}}{\pi}  + 2 Ce^{-\lambda M^\alpha}.
\end{align*}
Take $M = \left(\lambda^{-1}\ln(1/ \epsilon)\right)^{\frac{1}{\alpha}}$ to get
$$\E[F(X) - F(Y)] \leq 2 \epsilon \left[ C+ \frac{2^{p-1}C'\sqrt{10}K_{p+1}}{\pi}\sqrt{1+\left(\lambda^{-1}\ln(1/ \epsilon)\right)^{\frac{1}{\alpha}}}\right].$$
\end{proof}

As already mentioned in the Introduction, assumption (H2) is sometimes
better replaced by the following. 

\begin{theorem}\label{theor:bounds-char-funct-1}
Under the assumptions  of Theorem \ref{thm:bounds-char-funct-1} with $(H2)$ replaced by
\begin{itemize}
\item[(H2')] $\exists \gamma>0, \exists C>0, \forall A>0$, $\PP[|X|>A] \leq C A^{-\gamma}$ and $\PP[|Y|>A] \leq C A^{-\gamma}$,
\end{itemize}
and if $ \epsilon < \frac{\pi\sqrt{5}\gamma C}{5\cdot2^p C_p }$,
then
\eq
\mathcal{W}_{p+1} (X,Y) \leq \epsilon^{\frac{2\gamma}{2\gamma+1}}\left( \frac{2^{p+1}\sqrt{5M}C_p}{\pi}  \right)^{\frac{2\gamma}{2\gamma+1}} (2\gamma C)^{\frac{1}{2\gamma+1}}\left( 1+\frac{1}{2\gamma} \right).
\qe
\end{theorem}
\begin{proof}
The proof is similar to the one of Theorem \ref{thm:bounds-char-funct-1}; only the final optimization step (in $M$) changes. More precisely, for every test function $F  \in \mathcal C_c^\infty(\R)\cap  B_{0,p+1}$, and if $M>1$, we have
$$\E[F(X) - F(Y)] \leq \epsilon \frac{2^{p+1}\sqrt{5M}K_{p+1}}{\pi}  + 2 C M^{-\gamma}.$$
Taking $M = \left( \frac{\epsilon 2^{p}\sqrt{5}K_{p+1}}{\gamma C\pi} \right)^{-\frac{2}{2\gamma +1}}$ yields the result.
\end{proof}

\section{The Stein-Tikhomirov method for arbitrary target}
\label{sec:general-approach-st}

Let $\Delta_n(t) = \phi_n(t) - \phi_{\infty}(t)$ be as in the
Introduction. Our general method rests on the following elementary
observation: it is always possible to write, at least formally,
\begin{equation}
  \label{eq:9}
\Delta_n(t) :=   \phi_n(t) - \phi_{\infty}(t) = \left( \mathcal{D}_{\infty} \circ L_{\infty} \right) 
  \phi_n (t) 
\end{equation}
with $L_{\infty}$ the differential operator
\begin{equation}
  \label{eq:13}
   \phi \mapsto L_{\infty}\phi = \phi_\infty \left(  \phi / \phi_{\infty}
\right)' = \phi' - \frac{\phi_{\infty}'}{\phi_{\infty}} \phi (=
L_{\infty} \Delta_n)
\end{equation}
(the last equality because $L_{\infty}(\phi_{\infty}) = 0$, trivially)
and $\mathcal{D}_{\infty}$ the integral operator
\begin{equation}
  \label{eq:18}
\mathcal{D}_{\infty}f(t) =   \phi_{\infty}(t) \int_0^t
  \frac{1}{\phi_{\infty}(\xi)}f(\xi) d\xi.
\end{equation}
We propose to tackle the problem of estimating rates of convergence of
$|\Delta_n(t)|$ to 0 (and thus convergence in Kolmogorov via
\eqref{eq:1} or convergence in smooth Wasserstein \eqref{eq:6} via
Theorems~\ref{thm:bounds-char-funct-1} and
\ref{theor:bounds-char-funct-1}) by studying \eqref{eq:13} and
\eqref{eq:18}.

It is not hard to inscribe the original Stein-Tikhomirov method in the
above framework. Indeed, in the standard Gaussian case, differential
operator \eqref{eq:13} is
\[L_{\infty}\phi(t) = \phi'(t)+t\phi(t)\] and the integral operator
\eqref{eq:18} is
\begin{equation*}
  \mathcal{D}_{\infty}f(t) = e^{-t^2/2} \int_0^te^{\xi^2/2} f(\xi) d
  \xi. 
\end{equation*}
Tikhomirov's estimate \eqref{eq:5} can be directly integrated via
$\mathcal{D}_{\infty}$, an integral operator which moreover has
agreeable properties (see Figure \ref{fig:dawson} for the function
$\mathcal{D}_{\infty}1(t)$, known as the Dawson's function).
Similarly as for the classical Stein's method, we shall see in our
applications that the properties of \eqref{eq:13} and \eqref{eq:18}
remain favorable even under a change of target and approximating
sequence. More precisely, we shall show in the applications section
that under weak assumptions on $\phi_n$, the function
$L_{\infty}\phi_n$ admits polynomial upper bounds of the form
\begin{equation}
  \label{eq:12}
   \left| L_{\infty}\phi_n(t) \right| \le \Delta(F_n, F_{\infty})
\left|t \right|^{p}
\end{equation}
for some $p>0$ with $\Delta(F_n, F_{\infty}) $ an \emph{explicit}
quantity appearing naturally in Stein's method which captures
essential features of the proximity in distribution between the two
laws.  Similarly, operator $\mathcal{D}_{\infty}$ preserves these
bounds (although it sometimes increases the order of the polynomial
bound by 1) and allows, via \eqref{eq:9}, to obtain polynomial bounds
on $\Delta_n(t)$ which depend on $\Delta(F_n, F_{\infty})$. These in
turn will lead, via \eqref{eq:28} or the results from
Section~\ref{sec:bounds-char-funct}, to bounds in Kolmogorov or smooth
Wasserstein distances.

\subsection{The operator $L_{\infty}$ and its connection with Stein
  operators}
\label{sec:operator-l_infty}

The classical Stein identity states that a random variable
$F_{\infty}$ is standard Gaussian if and only if
\begin{equation}\label{eq:4}
  \mathbb{E} \left[ F_{\infty} f(F_{\infty}) \right] = \mathbb{E} \left[ f'(F_{\infty}) \right]
\end{equation}
for all $f$ for which both expectations exist. 
Plugging $f(\cdot) = e^{i t \cdot}$ into the above immediately leads,
by duality, to 
\begin{equation*}
  \frac{\phi_{\infty}'(t)}{\phi_{\infty}(t) }=  -t 
\end{equation*}
from which we read the operator $L_{\infty}$.  In other words
Tikhomirov's operator $L_{\infty}$ for the Gaussian is the dual, in
Fourier space, to Stein's operator
$\mathcal{A}_{\infty}f(x) = f'(x) - xf(x)$.  

There are many targets $F_{\infty}$ for which a similar ``Stein
identity / Stein-Tikhomirov operator'' duality may be applied. In
Section \ref{sec:quant-bounds-gener} we will focus on the so-called
Dickman distribution which is one of those positive $F_{\infty}$ which
for all bounded functions $f$ satisfy an additive biasing identity
\begin{equation}
  \label{eq:16}
  \mathbb{E} \left[ F_{\infty} f(F_{\infty}) \right]  = E \left[
    f(F_{\infty} + U) \right]
\end{equation}
for some $U$ independent of $F_{\infty}$ (Dickman distributed random
variables satisfy \eqref{eq:16} with $U$ a uniform random variable on
$[0, 1]$). Applying \eqref{eq:16} to $f(\cdot) = e^{i t \cdot}$ leads,
again by duality, to the differential equation on the CF
\begin{equation*}
\frac{\phi'_{\infty}(t)}{\phi_{\infty}(t)}= - i\phi_U(t) 
\end{equation*}
from which we read the operator 
\begin{equation*}
  L_{\infty}\phi(t) = \phi'(t) + i \phi_U(t) \phi(t). 
\end{equation*}
General conditions on $F_{\infty}$ under which biasing identities
\eqref{eq:16} hold are well known, and examples other than the Dickman
include infinitely divisible distributions which can be size biased,
compound Poisson random variables and many more (see
\cite{arratia2013size}).  In principle a version of our ST method can
be easily set up for targets in any of these families.

Identities \eqref{eq:4} and \eqref{eq:16} are but two examples of
so-called \emph{Stein identities} which have long been known to be
available for virtually any target distribution $F_{\infty}$ whose
density satisfies some form of differential/difference equation. A
general method for deriving Stein identities relies on the
construction of differential/difference operators $f \mapsto \mathcal{A}_{\infty}f$
such that 
\begin{equation*}
  \mathbb{E} \left[ \mathcal{A}_{\infty}f(F_{\infty}) \right] = 0
  \mbox{ for all } f \in \mathcal{F}(F_{\infty})
\end{equation*} 
with $\mathcal{F}(F_{\infty})$ some class of functions (see
\cite{LRS16}).  Clearly all targets which admit an operator of the
form $\mathcal{A}_{\infty}f(x) = \tau_1(x) f'(x) + \tau_2(x) f(x)$
such that $\tau_j, j=1, 2$ are polynomials of degree 1 (this includes,
see e.g.\ \cite{Stein1986}, many of the Pearson distributions) will
have an operator $L_{\infty}$ with agreeable structure.  More
generally any target which admits a Stein operator of the form
\begin{equation}
  \label{eq:14}
 \mathcal{A}_{\infty}f(x)  =  \sum_{j=0}^d (\alpha_j x + \beta_j) f^{(d)}(x)
\end{equation}
with $\alpha_j, \beta_j$ real numbers shall be a potentially good
candidate for our Stein-Tikhomirov approach. Indeed if
$\mathbb{E}\left[ \mathcal{A}_{\infty}f(F_{\infty}) \right]=0$ for
$\mathcal{A}_{\infty}$ as in \eqref{eq:14} then, working by duality in
Fourier space via test functions of the form
$f(\cdot) = e^{it \cdot}$, the fact that the coefficients of the
derivatives are polynomials guarantees that the corresponding CF will
be of the form
\begin{equation}
  \label{eq:15}
  \frac{\phi_{\infty}'(t)}{\phi_{\infty}(t)} = \frac{\sigma_B(t)}{\sigma_A(t)}
\end{equation}
with $\sigma_A(t) = \prod_{j=1}^{d_A}(t-\lambda^A_j)$ and
$\sigma_B(t) = \prod_{j=1}^{d_B}(t-\lambda^B_j)$ polynomials. We will
illustrate the power of this approach in Section
\ref{sec:bounds-targ-second} where we will apply the method to linear
combinations of independent centered chi-squared random variables,
leading to rates of convergence towards targets belonging to the
second Wiener chaos.

\subsection{A general  Stein-Tikhomirov method and generalized Dawson functions}
\label{sec:gener-appr-stein}

Similarly as for Stein's method whose robustness towards a change of
target has proven to be one of the strongpoints, there exist a great
variety of targets $F_{\infty}$ which admit a tractable operator
$L_{\infty}$. Building on this fact, the method we propose to adopt
can be broken down into the three following steps:
\begin{itemize}
\item Step 1: compute the operator
  $L_{\infty}\phi = L_{\infty}(\phi - \phi_{\infty})$ defined in
  \eqref{eq:13} and use properties of $\phi_{\infty}$ and of the
  approximating sequence $F_n$ to obtain bounds of the form
  \begin{equation}
    \label{eq:8}
\left| L_{\infty}(\phi_n -  \phi_{\infty})(t)   \right| \le  \Delta(F_n,
F_{\infty}) \psi(t)
  \end{equation}
  for $\Delta(F_n, F_{\infty})$ some relevant \emph{and} computable
  quantity and $\psi(\cdot)$ an explicit function.
\item  Step 2: plug \eqref{eq:8} into \eqref{eq:9} to write 
  \begin{equation*}
    \left| \phi_n(\xi) - \phi_{\infty}(\xi) \right| \le \Delta(F_n, F_{\infty}) \left|
      \phi_{\infty}(\xi) \right| \int_0^{\xi} \frac{1}{\left| \phi_{\infty}(t)
      \right|}  \psi(t) dt  =: D_{\infty}(\psi)(\xi)
  \end{equation*}
  and study properties of $ D_{\infty}(\psi)(\xi)$.
\item Step 3: use information from Steps 1 and 2 in combination with
  Esseen's inequality \eqref{eq:1} or Theorems
  \ref{thm:bounds-char-funct-1} and \ref{theor:bounds-char-funct-1} of
  Section \ref{sec:bounds-char-funct} to transfer the knowledge on the
  proximity of the CFs into knowledge on proximity of the
  distributions either in Kolmogorov distance or in an appropriate
  smooth Wasserstein.
\end{itemize}
As we shall see, in all our examples the function $\psi(t)$ from Step
1 is actually an increasing polynomial in its argument, so that
\begin{align*}
   \left|
      \phi_{\infty}(\xi) \right| \int_0^{\xi} \frac{1}{\left| \phi_{\infty}(t)
      \right|}  \psi(t) dt \le  \psi(\xi) \left|
      \phi_{\infty}(\xi) \right| \int_0^{\xi} \frac{1}{\left| \phi_{\infty}(t)
      \right|} dt =: \psi(\xi) D_{\infty}(\xi).
\end{align*}
By analogy with the Gaussian case we call the function $D_{\infty}$ a
(generalized) Dawson function for $F_{\infty}$. Its properties shall
be crucial for the quality of the final bounds. For instance if
$F_{\infty}$ is gamma distributed then the corresponding function is
proportional to $\xi$, and if $F_{\infty}$ is symmetric
$\alpha$-stable then $\phi_{\infty}(\xi) = e^{-|\xi|^{\alpha}}$ and
$D_{\infty}$ is easy to compute explicitly. Numerical illustrations
are provided in Figure~\ref{fig:dawson}.
\begin{figure}[!htb]
\minipage{0.32\textwidth}
  \includegraphics[width=\linewidth]{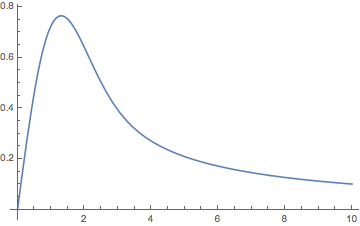}
\endminipage\hfill
\minipage{0.32\textwidth}
  \includegraphics[width=\linewidth]{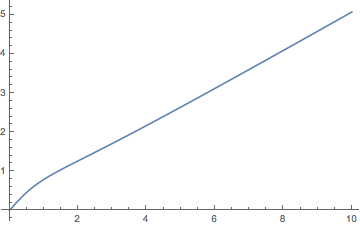}
\endminipage\hfill
\minipage{0.32\textwidth}%
  \includegraphics[width=\linewidth]{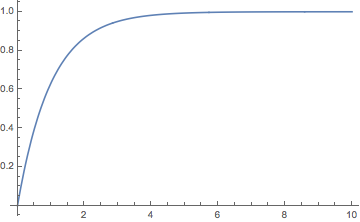}
\endminipage
\caption{ The function $D_{\infty}$ for $F_{\infty}$ (from left to
  right): standard Gaussian, Gamma$(1, 1)$ and L\'evy (symmetric
  $\alpha-$stable with $\alpha=1$) distributed.
}\label{fig:dawson}
\end{figure}
\vspace{+2cm}

\section{Bounds for targets in second Wiener chaos}
\label{sec:bounds-targ-second}
\subsection{Notations}
\noindent
Before stating the main results of this section, we introduce some notations and definitions from the Malliavin calculus. For any further details, we refer the reader to the books \cite{B98} and \cite{Nua}. Let $\{X(h)\}$ be an isonormal Gaussian process over a Hilbert space $\mathcal{H}$. Let $C^{\infty}_b(\mathbb{R}^n)$ be the space of infinitely differentiable and bounded functions on $\mathbb{R}^n$. We define the derivative operator by the following formula on regular random variables, $F=f(X(h_1),...,X(h_d))$:
\begin{align*}
\forall f\in C^{\infty}_b(\mathbb{R}^d),\ D(F)=\sum_{j=1}^d \frac{\partial}{\partial x_j}(f)(X(h_1),...,X(h_d))h_j.
\end{align*}
This operator is unbounded, closable and can be extended to the space $\mathbb{D}^{k,p}$ for $k\geq 1, p\geq 1$ (see e.g. section $1.2$ in \cite{Nua}). We denote by $\mathbb{D}^{\infty}=\bigcap_{k,p} \mathbb{D}^{k,p}$. Any square integrable random variable $F$ measurable with respect to the sigma algebra generated by $\{X(h)\}$ admits a Wiener-chaos expansion:
\begin{align*}
F=\sum_{n=0}^{+\infty}I_n(f_n),
\end{align*}
where the equality stands in $L^2(\Omega,\mathcal{F}(X),\mathbb{P})$ and with $f_n\in \mathcal{H}^{\hat{\otimes} n}$ and $I_n$ is the linear isometry between $\mathcal{H}^{\hat{\otimes} n}$ and the $n$-th Wiener chaos of $\{X(h)\}$. We have in particular the following isometry relation:
\begin{align*}
\operatorname{Var}\big(F\big)=\sum_{n=1}^{\infty}n! \mid\mid f_n\mid\mid^2_{ \mathcal{H}^{\hat{\otimes} n}}.
\end{align*}
Moreover, it is well known that $F$ belongs to $\mathbb{D}^{1,2}$ if and only if:
\begin{align*}
 \sum_{n=0}^\infty n n!  \mid\mid f_n\mid\mid^2_{ \mathcal{H}^{\hat{\otimes} n}}<+\infty
\end{align*}
We denote by $L$ the infinite dimensional Ornstein-Uhlenbeck operator whose domain is the space of random variables for which:
\begin{align*}
\sum_{n=0}^{+\infty}nI_n(f_n),
\end{align*}
is convergent in $L^2(\Omega,\mathcal{F}(X),\mathbb{P})$. We have:
\begin{align*}
\forall F\in\operatorname{Dom}(L),\ L(F)=-\sum_{n=0}^{\infty}nI_n(f_n).
\end{align*}
The pseudo-inverse of $L$ denoted by $L^{-1}$ acts as follows on centered random variable in $L^2(\Omega,\mathcal{F}(X),\mathbb{P})$:
\begin{align*}
L^{-1}(F)=\sum_{n=1}^{+\infty} -\frac{1}{n}I_n(f_n).
\end{align*}
To conclude, we introduce the iterated $\Gamma_j$ operators firstly defined in \cite{NourdinPeccati}.

\begin{defn}\label{Gamma}
For any $F\in \mathbb{D}^{\infty}$ and for any $j\geq 1$, we define $\Gamma_j(F)$ by the following recursive relation:
\begin{align*}
&\Gamma_0(F)=F,\\
&\Gamma_j(F)=\langle D(F),-DL^{-1}\Gamma_{j-1}(F)\rangle_{\mathcal{H}}.
\end{align*}
Note in particular that $\Gamma(F)=\Gamma_1(F):=\langle D(F),-DL^{-1}(F)\rangle_{\mathcal{H}}$. Moreover, we have:
\begin{align*}
\forall j\geq 1,\ \mathbb{E}[\Gamma_j(F)]=\frac{1}{j !} \kappa_{j+1}(F),
\end{align*}
where $\kappa_{j+1}(F)$ is the $j+1$ cumulants of $F$.
\end{defn}
\noindent
In the sequel, we consider a sequence $(F_n)$ lying in a finite sum of Wiener chaoses ($p$ Wiener chaoses). It is well known (\cite{Nua}) that such random variables are in $\mathbb{D}^{\infty}$. We emphasize that the existence of a finite integer $p$ such that the sequence lies in the sum of the $p$ first Wiener chaoses is for simplicity only. It is enough for our purpose and it prevents us from imposing conditions for carrying out interchange of derivative, integration and integration-by-parts. We denote its characteristic function by $\phi_{F_n}$.

\subsection{Results}
This section is devoted to limit theorems where the target lies in the
second Wiener chaos, i.e.\ we consider limiting distributions of the
following form:
\begin{equation}\label{Flimit}
  F = \sum_{i=1}^{m_1}\lambda_1(Z_i^2-1)+\sum_{i=m_1+1}^{m_1+m_2}\lambda_2(Z_i^2-1)+...+\sum_{i=m_1+\ldots+m_{d-1}+1}^q\lambda_d(Z_i^2-1),
\end{equation}
with $d \geq 1$, $q\geq 1$, $(m_1,..., m_d) \in \mathbb{N}^d$ such
that $m_1+...+m_d=q$, $(\lambda_1,..., \lambda_d) \in \mathbb{R}^*$
pairwise distinct and $(Z_i)_{i\geq 1}$ an i.i.d. sequence of standard
normal random variables. We denote its characteristic function by $\phi_\infty$. To start with we recall the following corollary from \cite{AAPS}:

\begin{corollary}\label{ba1}
Let $F$ be as in (\ref{Flimit}). Let $Y$ be a real valued random variable such that $\mathbb{E}[|Y|]<+\infty$. Then $Y \stackrel{\text{law}}{=} F$ if and only if, for all $\phi\in S(\mathbb{R})$, the space of rapidly decreasing infinitely differentiable functions:
\begin{align}
&\mathbb{E} \bigg[ \big(Y+\sum_{i=1}^d\lambda_im_i\big)(-1)^d2^d\bigg(\prod_{j=1}^d\lambda_j\bigg)\phi^{(d)}(Y)+\sum_{l=1}^{d-1}2^l(-1)^l\bigg(Ye_{l}(\lambda_1,...,\lambda_d)\nonumber\\ 
&+\sum_{k=1}^d\lambda_km_k\left(e_{l}(\lambda_1,...,\lambda_d)-e_{l}((\underline{\lambda}_k)\right)\bigg)\phi^{(l)}(Y)+Y\phi(Y) \bigg]=0,\label{eq:3}
\end{align}
with $e_{l}(\lambda_1,...,\lambda_d)$ is the elementary symmetric polynomials of degree $l$ in the variable $\lambda_1,...,\lambda_d$ and $\underline{\lambda}_k$ is the $k-1$-tuple $(\lambda_1,...,\lambda_{k-1},\lambda_{k+1},...,\lambda_d)$.
\end{corollary}
\noindent
Let us denote by $\mathcal{A}_{\infty}$ the differential operator appearing in the left-hand side of (\ref{eq:3}). By the proof of Theorem 2.1 of \cite{AAPS}, we can rewrite it in the following form:
\begin{align*}
\forall\phi\in S\big(\mathbb{R}\big),\ \mathcal{A}_{\infty}(\phi)=(x+<m,\lambda>)\mathcal{A}_{d,\lambda}(\phi)(x)-\mathcal{B}_{d,m,\lambda}(\phi)(x),
\end{align*}
with,
\begin{align*}
&\mathcal{A}_{d,\lambda}(\phi)(x)=\frac{1}{2\pi}\int_{\mathbb{R}}\mathcal{F}(\phi)(\xi)\bigg(\prod_{k=1}^d(\frac{1}{2\lambda_k}-i\xi)\bigg)\exp(ix\xi)d\xi,\\
&\mathcal{B}_{d,m,\lambda}(\phi)(x)=\frac{1}{2\pi}\int_{\mathbb{R}}\mathcal{F}(\phi)(\xi)\bigg(\sum_{k=1}^d\frac{m_k}{2}\prod_{l=1,l\ne k}^d(\frac{1}{2\lambda_l}-i\xi)\bigg)\exp(ix\xi)d\xi,\\
&\mathcal{F}(\phi)(\xi)=\int_{\mathbb{R}}\phi(x)\exp(-ix\xi)dx.
\end{align*}
\noindent
We denote by $\sigma_{\mathcal{A}}$ and $\sigma_{\mathcal{B}}$ the
symbols associated with $\mathcal{A}_{d,\lambda}$ and
$\mathcal{B}_{d,m,\lambda}$ respectively, and defined by 
\begin{align*}
&   \sigma_{\mathcal{A}}(\xi) =
  \prod_{k=1}^d(\frac{1}{2\lambda_k}-i\xi) \\
  & \sigma_{\mathcal{B}}(\xi) =
    \sum_{k=1}^d\frac{m_k}{2}\prod_{l=1,l\ne
    k}^d(\frac{1}{2\lambda_l}-i\xi). 
\end{align*}
These symbols are infinitely differentiable and non-zero everywhere on
$\mathbb{R}$. Moreover, still by the proof of Theorem $2.1$ of
\cite{AAPS}, we have
\begin{align*}
  \frac{\phi_\infty'(\xi)}{\phi_\infty(\xi)}=\frac{-i<m,\lambda>\sigma_{\mathcal{A}}(\xi)+i\sigma_{\mathcal{B}}(\xi)
  }{\sigma_{\mathcal{A}}(\xi)} 
\end{align*}
(as was anticipated in Section \ref{sec:operator-l_infty}).  This
leads to the following result.


\begin{proposition}\label{formulaSC}
For any $\xi$ in $\mathbb{R}$, we have:
\begin{align}
\phi_{F_n}(\xi)-\phi_\infty(\xi)=i\phi_\infty(\xi)\int_0^{\xi}\frac{1}{\phi_\infty(s)}
  \frac{\mathbb{E}\big[\mathcal{A}_{\infty}(\exp(is.))(F_n)\big]}{\sigma_{\mathcal{A}}(s)}ds. \label{repdiffchar}
\end{align}
\end{proposition}

\begin{proof}
Let $\phi$ be a smooth enough function which will be specified later on. We are interested in the quantity $\mathbb{E}\big[\mathcal{A}_{\infty}(\phi)(F_n)\big]$:
\begin{align*}
 \mathbb{E}\big[\mathcal{A}_{\infty}(\phi)(F_n)\big]=&\mathbb{E}\big[F_n\mathcal{A}_{d,\lambda}(\phi)(F_n)\big]+<m,\lambda>\mathbb{E}\big[\mathcal{A}_{d,\lambda}(\phi)(F_n)\big]-\mathbb{E}\big[\mathcal{B}_{d,m,\lambda}(\phi)(F_n)\big]\\
=&\frac{1}{2\pi}\bigg[<\mathcal{F}\big(x\mathcal{A}_{d,\lambda}(\phi)\big);\phi_{F_n}>+<m,\lambda><\mathcal{F}\big(\mathcal{A}_{d,\lambda}(\phi)\big);\phi_{F_n}>\\
&-<\mathcal{F}\big(\mathcal{B}_{d,m,\lambda}(\phi)\big);\phi_{F_n}>\bigg],
\end{align*}
Note that the previous brackets can be understood as duality brackets as soon as $F_n$ as enough moments and $\mathcal{F}\big(\mathcal{A}_{d,\lambda}(\phi)\big)$ and $\mathcal{F}\big(\mathcal{B}_{d,m,\lambda}(\phi)\big)$ are distributions with compact support with orders less or equal to the number of moments of $F_n$.  We have:
\begin{align*}
\mathbb{E}\big[\mathcal{A}_{\infty}(\phi)(F_n)\big]=&\frac{1}{2\pi}\bigg[-i<\mathcal{F}(\phi);\sigma_{\mathcal{A}}(.)(\phi_{F_n})'>+<m,\lambda><\mathcal{F}(\phi);\sigma_{\mathcal{A}}(.)\phi_{F_n}>\\
&-<\mathcal{F}(\phi);\sigma_{\mathcal{B}}(.)\phi_{F_n}>\bigg].
\end{align*}
Thus, we have:
\begin{align*}
\mathbb{E}\big[\mathcal{A}_{\infty}(\phi)(F_n)\big]&=\frac{-i}{2\pi}\bigg[<\mathcal{F}(\phi);\sigma_{\mathcal{A}}(.)(\phi_{F_n})'>-<\mathcal{F}(\phi);\sigma_{\mathcal{A}}(.)\frac{(\phi_\infty)'}{\phi_\infty}\phi_{F_n}>\bigg],\\
&=\frac{-i}{2\pi}<\mathcal{F}(\phi);\sigma_{\mathcal{A}}(.)\phi\big(\frac{\phi_{F_n}}{\phi_\infty}\big)'>.
\end{align*}
Now choose $\phi$ equal to $\exp(i\xi.)$. Thus, its Fourier transform is exactly $2\pi$ times the Dirac distribution at $\xi$. We obtain:
\begin{align*}
\mathbb{E}\big[\mathcal{A}_{\infty}(\exp(i\xi.))(F_n)\big]&=(-i)\sigma_{\mathcal{A}}(\xi)\phi_\infty(\xi)\big(\frac{\phi_{F_n}(\xi)}{\phi_\infty(\xi)}\big)',\\
&=(-i)\sigma_{\mathcal{A}}(\xi)\phi_\infty(\xi)\big(\frac{\phi_{F_n}(\xi)-\phi_\infty(\xi)}{\phi_\infty(\xi)}\big)'.
\end{align*}
This ends the proof of the proposition.
\end{proof}
\noindent
In order to use identity \eqref{repdiffchar} for our Stein-Tikhomirov
method for the second Wiener chaos, it still remains to show that the
integrand has good properties. To this end we rely on the papers
\cite{AAPS} (page $12$ equation $2.9$) and \cite{a-p-p} (Proposition
$3.2$ page $13$), where it is proved that
$\mathbb{E}\big[\mathcal{A}_{\infty}(\exp(i\xi.))(F_n)\big]$ admits
the following suitable representation in terms of iterated gamma
operators. Regarding notations and definitions, we refer to \cite{AAPS}, (at the beginning
of Section $2.3$ page $10$).

\begin{equation}\label{eq:repMal}
\begin{split}
\E \left[ \mathcal{A}_\infty \big(\phi\big) (F_n) \right] & = - \E \Bigg[  \phi^{(d)}(F_n) \times \Big( \sum_{r=1}^{d+1} a_r \left[ \Gamma_{r-1}(F_n) - \E[\Gamma_{r-1}(F_n)] \right] \Big) \Bigg] \\
& \hskip 1cm +  \sum_{r=2}^{d+1} a_r \sum_{l=0}^{r-2}  \left\{ \E [
  \phi^{(d-l)}(F_n) ] \times \Big(  \E \left[ \Gamma_{r-l-1}(F)
  \right] - \E \left[ \Gamma_{r-l-1}(F_n) \right] \Big) \right\}\\ 
& = - \E \Bigg[  \phi^{(d)}(F_n) \times \Big( \sum_{r=1}^{d+1} a_r \left[
  \Gamma_{r-1}(F_n) - \E[\Gamma_{r-1}(F_n)] \right] \Big) \Bigg] \\ 
&\hskip 1cm +  \sum_{r=2}^{d+1} a_r \sum_{l=0}^{r-2} \frac{ \E [ \phi^{(d-l)}(F_n) ]}{(r-l-1)!} \times \Big( \kappa_{r-l}(F) - \kappa_{r-l}(F_n) \Big).
\end{split}
\end{equation}
Combining this expression and the formula (\ref{repdiffchar}) for the difference of the characteristic functions of $F_n$ and $F$, it holds that $F_n$ converges in distribution towards $F$ if the following conditions are satisfied:
\begin{align*}
\forall r=2,\ldots,d+1,\ \kappa_{r}(F_n)\rightarrow\kappa_{r}(F)\\
  \sum_{r=1}^{d+1} a_r [ \Gamma_{r-1}(F_n) - \E[\Gamma_{r-1}(F_n)]\rightarrow 0.
\end{align*}
Our goal is now to give a quantitative bound. Firstly note that:
\begin{align*}
\vert\E \left[ \mathcal{A}_\infty (e^{i\xi.}) (F_n) \right]\vert \leq &\vert\xi\vert^d\E \Bigg[\vert\sum_{r=1}^{d+1} a_r \left[
  \Gamma_{r-1}(F_n) - \E[\Gamma_{r-1}(F_n)] \right]\vert \Bigg]\\
  &+\sum_{r=2}^{d+1} \vert a_r\vert \sum_{l=0}^{r-2} \frac{ \vert\xi\vert^{d-l}}{(r-l-1)!} \times \vert \kappa_{r-l}(F) - \kappa_{r-l}(F_n) \vert\\
  &\leq (1+\vert\xi\vert)^d\Delta_n,
\end{align*}
with:
\begin{align*}
\Delta_n=\E \Bigg[\vert\sum_{r=1}^{d+1} a_r \left[
  \Gamma_{r-1}(F_n) - \E[\Gamma_{r-1}(F_n)] \right]\vert \Bigg]+\sum_{r=2}^{d+1} \vert a_r\vert \sum_{l=0}^{r-2} \frac{1}{(r-l-1)!} \times \vert \kappa_{r-l}(F) - \kappa_{r-l}(F_n) \vert.
\end{align*}
Thus, we obtain:
\begin{align}
\vert\phi_{F_n}(\xi)-\phi_\infty(\xi)\vert\leq \Delta_n \vert\phi_\infty(\xi)\vert\int_0^{\xi}\dfrac{(1+\vert s\vert)^d}{\vert\sigma_{\mathcal{A}}(s)\phi_\infty(s)\vert}ds,\label{bound1}
\end{align}
Before moving on, we need the following simple lemma.
\begin{lemma}\label{technic}
Let F be as above and $\phi_\infty$ its characteristic function. We have:
\begin{align*}
&\forall\xi\in\mathbb{R},\ \phi_\infty(\xi)=\exp(-i\xi<m,\lambda>)\prod_{k=1}^d\dfrac{1}{(1-2i\xi\lambda_k)^{\frac{m_k}{2}}},\\
&\forall\xi\in\mathbb{R},\ \dfrac{1}{(1+4\lambda_{max}^2\xi^2)^{\frac{q}{4}}}\leq\vert \phi_\infty(\xi) \vert \leq \dfrac{1}{(1+4\lambda_{min}^2\xi^2)^{\frac{q}{4}}},\\
&\forall\xi\in\mathbb{R},\  \frac{1}{\vert \sigma_{\mathcal{A}}(s)\vert} \leq \dfrac{2^d\prod_{j=1}^d\vert\lambda_{j}\vert}{(1+4\lambda_{min}^2\xi^2)^{\frac{d}{2}}}.
\end{align*}
\end{lemma}
\begin{proof}
Based on the representation of $F$, the proof is standard. We leave the details to the interested reader.
\end{proof}
\noindent
Using these inequalities into (\ref{bound1}), we obtain:
\begin{align}
&\vert\phi_{F_n}(\xi)-\phi_\infty(\xi)\vert\leq \Delta_n \dfrac{2^d\prod_{j=1}^d\vert\lambda_{j}\vert}{(1+4\lambda_{min}^2\xi^2)^{\frac{q}{4}}}\int_0^{\xi}\dfrac{(1+\vert s\vert)^d (1+4\lambda_{max}^2s^2)^{\frac{q}{4}}}{(1+4\lambda_{min}^2s^2)^{\frac{d}{2}}}ds\nonumber,\\
&\vert\phi_{F_n}(\xi)-\phi_\infty(\xi)\vert\leq C_{d,\lambda}\Delta_n \int_0^{\xi}\dfrac{(1+\vert s\vert)^d}{(1+s^2)^{\frac{d}{2}}}ds\nonumber,\\
&\vert\phi_{F_n}(\xi)-\phi_\infty(\xi)\vert\leq C_{d,\lambda}\Delta_n \vert\xi\vert.\label{bound2}
\end{align}
Now, we use the bound (\ref{bound2}) to obtain a quantitative bound on smooth Wasserstein metric between $F_n$ and $F$. Moreover, using Ess\'een inequality (\ref{eq:1}), we obtain a quantitative bound in Kolmogorov distance thanks to (\ref{bound2}) for $q\geq 3$. This is the content of the next Theorem.

\begin{theorem}\label{bound-secondchaos}
Let $F_n$ be a sequence of random variables lying in the finite sum of the $p$ first Wiener chaoses and let 
\begin{equation*}
F = \sum_{i=1}^{m_1}\lambda_1(N_i^2-1)+\sum_{i=m_1+1}^{m_1+m_2}\lambda_2(N_i^2-1)+...+\sum_{i=m_1+\ldots+m_{d-1}+1}^q\lambda_d(N_i^2-1),
\end{equation*}
for parameters $(\lambda_i)_{1\le i \le d}$ and $m_1,\cdots,m_{d} \in \N^{d}$ as in Corollary \ref{ba1}. Then, there exist some constants $C>0$ and $\theta>0$ depending only on the target and $p$ such that
\begin{equation*}
\mathcal W_2( F_n,F) \le C \Delta_n \mid\log( \Delta_n)\mid^\theta,
\end{equation*}
with
\begin{align*}
\Delta_n&=\E \Bigg[\vert\sum_{r=1}^{d+1} a_r \left[
  \Gamma_{r-1}(F_n) - \E[\Gamma_{r-1}(F_n)] \right]\vert \Bigg]+\sum_{r=2}^{d+1} \vert a_r\vert \sum_{l=0}^{r-2} \frac{1}{(r-l-1)!}\\
  & \times \vert \kappa_{r-l}(F) - \kappa_{r-l}(F_n) \vert.
\end{align*} 
Moreover, for $q\geq 3$, there exists a strictly positive constant $B>0$ depending on the limiting distribution $F$ through $d,q$ and the $\lambda_i$, such that:
\begin{align*}
\operatorname{Kol}\big(F_n,F\big)\leq B\sqrt{\Delta_n}.
\end{align*}
\end{theorem}
\begin{proof}
The first part of the theorem is a direct consequence of (\ref{bound2}) and Theorem \ref{thm:bounds-char-funct-1} and the fact that the sequence $(F_n)$ lies in a finite sum of the $p$ first Wiener chaoses. Let $q\geq 3$ and $\lambda_{\operatorname{min}}\ne 0$. Thanks to the second point of Lemma \ref{technic}, the density of the limiting distribution is bounded on $\mathbb{R}$ so that we can apply Ess\'een inequality (\ref{eq:1}). For any $T>0$, we have (thanks to (\ref{bound2})):
\begin{align*}
\operatorname{Kol}\big(F_n,F\big)\leq B\big(\Delta_n T+\frac{1}{T}\big).
\end{align*}
Now, we set $T:=1/\sqrt{\Delta_n}$. This concludes the proof.
\end{proof}
\noindent
The quantitative bound obtained in Theorem \ref{bound-secondchaos} must be compared with analogous estimates for the convergence in distribution for sequences of random variable belonging to a fixed Wiener chaos towards the non-central gamma distribution of parameter $\nu>0$ as considered in \cite{NP09} and in \cite{NP09II} (see e.g. Theorems $1.5$ and $3.11$ of \cite{NP09II}). Then, we have the following Lemma.

\begin{lemma}\label{ncgamma}
Let $\nu>0$, $F=\nu(Z^2-1)$ and $F_n=I_r(f_n)$ be a sequence of random variables belonging to the $r$-th Wiener chaos. Then, we have:
\begin{align*}
\Delta_n&\leq \E \Bigg[ (\nu^2-\frac{1}{2}r!\mid\mid f_n\mid\mid^2_{\mathcal{H}^{\otimes r}})F_n^2-\frac{\nu}{2}F^3_n+\frac{1}{4r^2}\mid\mid D(F_n)\mid\mid^4_{\mathcal{H}}+\frac{(r!)^2}{4}\mid\mid f_n\mid\mid^4_{\mathcal{H}^{\otimes r}} \Bigg]^{\frac{1}{2}}\\
&+ \vert \nu^2-\frac{r!}{2}\mid\mid f_n\mid\mid_{\mathcal{H}^{\otimes r}}^2 \vert.
\end{align*}
If $r=2$, we have the moment bound:
\begin{align*}
\Delta_n\leq \E \Bigg[ (\nu^2-\frac{1}{2}\mathbb{E}\big[F_n^2\big])F_n^2+\big(4\nu^4-\frac{\nu}{2}F^3_n)+\big(-\frac{5}{2}v^4+\frac{1}{24}F_n^4\big)+\frac{3}{8}\big((\mathbb{E}\big[F_n^2\big])^2-4\nu^4\big) \Bigg]^{\frac{1}{2}}+ \vert \nu^2-\frac{1}{2}\mathbb{E}\big[F_n^2\big] \vert.
\end{align*}
\end{lemma}

\begin{proof}
We are interested in the proximity between $F_n=I_r(f_n)$ and $F=\nu(Z^2-1)$. Then, $\Delta_n$ reduces to:
\begin{align}\label{SimEx}
\Delta_n&=\E \Bigg[\vert
 a_1I_r(f_n)+a_2 \Gamma\big(I_d(f_n)\big)-a_2\kappa_{2}(I_r(f_n))\vert \Bigg]+\vert a_2\vert \vert \kappa_{2}(\nu(Z^2-1)) - \kappa_{2}(I_r(f_n)) \vert,\nonumber\\
 &=\E \Bigg[\vert
 a_1I_r(f_n)+a_2 \Gamma\big(I_r(f_n)\big)-a_2\kappa_{2}(I_r(f_n))\vert \Bigg]+\vert a_2\vert \vert 2\nu^2-r!\mid\mid f_n\mid\mid_{\mathcal{H}^{\otimes r}}^2 \vert.
\end{align}
We are going to give a bound on the first term of the right hand side of equation (\ref{SimEx}). First note that $a_1=-\nu$ and $a_2=1/2$. By Cauchy-Schwarz inequality, we have:
\begin{align*}
\E \Bigg[\vert
 a_1I_r(f_n)+a_2 \Gamma\big(I_r(f_n)\big)-a_2\kappa_{2}(I_r(f_n))\vert \Bigg] \leq \E \Bigg[\vert
 -\nu F_n+\frac{1}{2r}\mid\mid D(F_n)\mid\mid^2_{\mathcal{H}} -\frac{1}{2}r! \mid\mid f_n\mid\mid^2_{\mathcal{H}^{\otimes r}}\vert^2 \Bigg]^{\frac{1}{2}}.
\end{align*}
Expanding the square, we have:
\begin{align*}
\E \Bigg[\vert
 -\nu F_n+\frac{1}{2r}\mid\mid D(F_n)\mid\mid^2_{\mathcal{H}} -\frac{1}{2}r! \mid\mid f_n\mid\mid^2_{\mathcal{H}^{\otimes r}}\vert^2 \Bigg]&=\E \Bigg[ \nu^2F_n^2-\frac{\nu}{r}F_n\mid\mid D(F_n)\mid\mid^2_{\mathcal{H}}\\
 &+\frac{1}{4r^2}\mid\mid D(F_n)\mid\mid^4_{\mathcal{H}}+\nu F_nr!\mid\mid f_n\mid\mid^2_{\mathcal{H}^{\otimes r}}\\
 &-\frac{1}{2r}r!\mid\mid f_n\mid\mid^2_{\mathcal{H}^{\otimes r}}\mid\mid D(F_n)\mid\mid^2_{\mathcal{H}}+\frac{(r!)^2}{4}\mid\mid f_n\mid\mid^4_{\mathcal{H}^{\otimes r}} \Bigg].
\end{align*}
Now, we use Lemma $2.1$ of \cite{NP09II} to get:
\begin{align*}
\E \Bigg[\vert
 -\nu F_n+\frac{1}{2r}\mid\mid D(F_n)\mid\mid^2_{\mathcal{H}} -\frac{1}{2}r! \mid\mid f_n\mid\mid^2_{\mathcal{H}^{\otimes r}}\vert^2 \Bigg]&=\E \Bigg[ \nu^2F_n^2-\frac{\nu}{2}F^3_n+\frac{1}{4r^2}\mid\mid D(F_n)\mid\mid^4_{\mathcal{H}}\\
 &-\frac{1}{2}r!\mid\mid f_n\mid\mid^2_{\mathcal{H}^{\otimes r}}F_n^2+\frac{(r!)^2}{4}\mid\mid f_n\mid\mid^4_{\mathcal{H}^{\otimes r}} \Bigg]
\end{align*}
Thus, we have the following bound for $\Delta_n$:
\begin{align*}
\Delta_n\leq \E \Bigg[ (\nu^2-\frac{1}{2}r!\mid\mid f_n\mid\mid^2_{\mathcal{H}^{\otimes r}})F_n^2-\frac{\nu}{2}F^3_n+\frac{1}{4r^2}\mid\mid D(F_n)\mid\mid^4_{\mathcal{H}}+\frac{(r!)^2}{4}\mid\mid f_n\mid\mid^4_{\mathcal{H}^{\otimes r}} \Bigg]^{\frac{1}{2}}+ \vert \nu^2-\frac{r!}{2}\mid\mid f_n\mid\mid_{\mathcal{H}^{\otimes r}}^2 \vert.
\end{align*}
In particular, if $F_n$ is in the second Wiener chaos, using Lemma $7.1$ of \cite{NP09II}, we obtain the moment bound:
\begin{align*}
\Delta_n\leq \E \Bigg[ (\nu^2-\mid\mid f_n\mid\mid^2_{\mathcal{H}^{\otimes 2}})F_n^2-\frac{\nu}{2}F^3_n+\frac{1}{24}F_n^4+\frac{3}{2}\mid\mid f_n\mid\mid^4_{\mathcal{H}^{\otimes 2}} \Bigg]^{\frac{1}{2}}+ \vert \nu^2-\mid\mid f_n\mid\mid_{\mathcal{H}^{\otimes 2}}^2 \vert,
\end{align*}
which concludes the proof of the corollary.
\end{proof}

\subsection{An application to $U$-statistics}
\label{sec:an-application-u}

A classical setting for which we can obtain explicit rates of convergence is taken from the theory of $U$-statistics. We consider a second order $U$-statistics with kernel $h$ defined by:
\begin{align*}
h(x,y)=\alpha H_{q}(x)H_{q}(y),
\end{align*}
where $H_{q}$ is the Hermite polynomial of degree $q$ and $\alpha\ne 0$. Then, we set:
\begin{align*}
nU_n(h)&=\frac{2}{n-1}\sum_{i<j}h(Z_i,Z_j),\\
&=\frac{\alpha}{q!}I_{2q}\bigg(\frac{2}{n-1}\sum_{i<j}h_i^{\otimes q} \tilde{\otimes} h_j^{\otimes q}\bigg).
\end{align*}
By the standard theory of U-statistics, we have the following convergence in distribution:
\begin{align*}
nU_n(h) \Rightarrow \alpha (Z^2-1).
\end{align*}
In order to apply the bound of Theorem \ref{bound-secondchaos}, we need to compute $\Gamma(nU_n(h))$ and the second cumulant of $nU_n(h)$. First of all, we have:
\begin{align*}
\kappa_2\big(nU_n(h)\big)&=\mathbb{E}\big(n^2 U_n(h)^2\big),\\
&=\frac{\alpha^2}{(q!)^2}(2q)! \big(\frac{2}{n-1}\big)^2\sum_{i_1<j_1}\sum_{i_2<j_2}\langle h_{i_1}^{\otimes q}\tilde{\otimes} h_{j_1}^{\otimes q}; h_{i_2}^{\otimes q}\tilde{\otimes} h_{j_2}^{\otimes q}\rangle_{\mathcal{H}^{\otimes 2q}},\\
&=\frac{1}{4}\frac{\alpha^2}{(q!)^2}(2q)! \big(\frac{2}{n-1}\big)^2\sum_{i_1\ne j_1}\sum_{i_2\ne j_2}\langle h_{i_1}^{\otimes q}\tilde{\otimes} h_{j_1}^{\otimes q}; h_{i_2}^{\otimes q}\tilde{\otimes} h_{j_2}^{\otimes q}\rangle_{\mathcal{H}^{\otimes 2q}}.
\end{align*}
Fix $i_1\ne j_1$ and $i_2\ne j_2$. Using the definition of the symmetrization operator, we have:
\begin{align*}
\langle h_{i_1}^{\otimes q}\tilde{\otimes} h_{j_1}^{\otimes q}; h_{i_2}^{\otimes q}\tilde{\otimes} h_{j_2}^{\otimes q}\rangle_{\mathcal{H}^{\otimes 2q}}&=\langle h_{i_1}^{\otimes q}\otimes h_{j_1}^{\otimes q}; h_{i_2}^{\otimes q}\tilde{\otimes} h_{j_2}^{\otimes q}\rangle_{\mathcal{H}^{\otimes 2q}}.
\end{align*}
Now, we note that if $i_1\ne i_2\ne j_2$ then the previous quantity is equal to $0$ and similarly if $j_1\ne i_2\ne j_2$. Thus, we have to consider only the cases $i_1=i_2$, $j_1=j_2$ and $i_1=j_2$, $j_1=i_2$. Assume that $i_1=i_2$ and $j_1=j_2$. In this case, the only permutations which are going to contribute are the one such that:
\begin{align*}
&\forall i\in\{1,...,q \},\ \sigma(i)\in\{1,...,q\},\\
&\forall j\in\{q+1,...,2q \},\ \sigma(j)\in\{q+1,...,2q\}.
\end{align*}
We have $(q!)^2$ such permutations. Thus, we have:
\begin{align*}
\kappa_2\big(nU_n(h)\big)&=\frac{1}{4}\frac{\alpha^2}{(q!)^2}(2q)! \big(\frac{2}{n-1}\big)^2\sum_{i_1\ne j_1}\sum_{i_2\ne j_2}\bigg(\delta_{i_1,i_2}\delta_{j_1,j_2}+\delta_{j_1,i_2}\delta_{i_1,j_2}\bigg)\frac{(q!)^2}{(2q)!},\\
&=\alpha^2 \big(\frac{2}{n-1}\big)^2\bigg[ \sum_{i_1\ne j_1} 1\bigg],\\
&=\alpha^2 \big(\frac{2}{n-1}\big)^2\dfrac{n(n-1)}{2},\\
&=\alpha^2\frac{2 n}{n-1}.
\end{align*}
Thus, we can write for $\Delta_n$:
\begin{equation}
\Delta_n=\mathbb{E}\big(\vert a_1(nU_n(h))+a_2\Gamma(nU_n(h))-a_2 \kappa_{2}\big(nU_n(h)\big)  \vert\big)+ \vert a_2\vert \vert \kappa_{2}\big(\alpha (Z^2-1)\big) - \kappa_{2}\big(nU_n(h)\big) \vert.
\end{equation}
In particular, for the second term on the right-hand side, we have the following quantitative bound:
\begin{align*}
\vert a_2\vert \vert \kappa_{2}\big(\alpha (Z^2-1)\big) - \kappa_{2}\big(nU_n(h)\big) \vert \leq\frac{2\alpha\vert a_2\vert}{n-1}.
\end{align*}
Now, for the first term we can use Theorem $8.4.4$ of \cite{NourdinPeccatibook} to obtain the following bound:
\begin{align*}
& \mathbb{E}\big(\vert a_1(nU_n(h))+a_2\Gamma(nU_n(h))-a_2
  \kappa_{2}\big(nU_n(h)\big)  \vert\big)\\
&=\mathbb{E}\bigg[\mid a_1 I_{2q}(f_n)+a_2\sum_{r=1}^{2q-1}(2q)(r-1)!\binom{2q-1}{r-1}^2 I_{4q-2r}\big(f_n\otimes_rf_n\big)\mid\bigg]\\
&\leq \mathbb{E}\bigg[\mid\sum_{r=1,r\ne q}^{2q-1}q(r-1)!\binom{2q-1}{r-1}^2 I_{4q-2r}\big(f_n\otimes_rf_n\big)\\
&\quad +I_{2q}(-\alpha f_n+q!\binom{2q-1}{q-1}^2f_n\otimes_{q}f_n)\mid\bigg].\\
&\leq \bigg(\sum_{r=1,r\ne q}^{2q-1}q^2(r-1)!^2\binom{2q-1}{r-1}^4(4q -2r)! \mid\mid f_n\tilde{\otimes}_rf_n\mid\mid^2_{\mathcal{H}^{\otimes 4q-2r}}\\
&\quad +(2q)!\mid\mid -\alpha f_n+q!\binom{2q-1}{q-1}^2f_n\tilde{\otimes}_{q}f_n\mid\mid^2_{\mathcal{H}^{\otimes 2q}}\bigg)^{\frac{1}{2}},
\end{align*}
where:
\begin{align*}
f_n=\frac{\alpha}{q!}\frac{2}{n-1}\sum_{i<j}h_i^{\otimes q} \tilde{\otimes} h_j^{\otimes q},
\end{align*}
and $f_n\tilde{\otimes}_rf_n$ denotes the symmetrized contraction of order $r$ of $f_n$ (see e.g. p $10$ of \cite{Nua} for a definition). To exhibit an explicit rate of convergence for this term, we restrict our study to the case where $q=1$. Therefore, the upper bound becomes:
\begin{align*}
\mathbb{E}\big(\vert a_1(nU_n(h))+a_2\Gamma(nU_n(h))-a_2 \kappa_{2}\big(nU_n(h)\big)  \vert\big)\leq \sqrt{2}\mid\mid -\alpha f_n+f_n\tilde{\otimes}_{1}f_n\mid\mid_{\mathcal{H}^{\otimes 2}}.
\end{align*}
Thus, we need to compute the following quantities:
\begin{align*}
&(I)=\langle f_n;f_n\otimes_1 f_n\rangle_{\mathcal{H}^{\otimes 2}},\\
&(II)=\langle f_n\tilde{\otimes}_1 f_n;f_n\tilde{\otimes}_1 f_n\rangle_{\mathcal{H}^{\otimes 2}}.
\end{align*}
First, we note that:
\begin{align*}
f_n\otimes_1 f_n=\bigg(\dfrac{2\alpha}{n-1}\bigg)^2\frac{1}{4}\sum_{i\ne j,\ l\ne k}h_i\tilde{\otimes}h_j\otimes_1 h_l\tilde{\otimes}h_k
\end{align*}
Moreover, we have:
\begin{align*}
h_i\tilde{\otimes}h_j\otimes_1 h_l\tilde{\otimes}h_k=\frac{1}{4}\big(h_i\otimes h_l\delta_{j,k}+h_i\otimes h_k\delta_{j,l}+h_j\otimes h_l\delta_{i,k}+h_j\otimes h_k\delta_{i,l}\big).
\end{align*}
Thus, we get:
\begin{align*}
(I)&=\frac{1}{4}\dfrac{\alpha^2}{(n-1)^2}\sum_{i\ne j,\ l\ne k} \langle f_n;h_i\otimes h_l\delta_{j,k}+h_i\otimes h_k\delta_{j,l}+h_j\otimes h_l\delta_{i,k}+h_j\otimes h_k\delta_{i,l} \rangle_{\mathcal{H}^{\otimes 2}},\\
&=\frac{1}{4}\dfrac{\alpha^3}{(n-1)^3}\sum_{i\ne j,\ l\ne k,\ m\ne o} \langle h_m\tilde{\otimes} h_o;h_i\otimes h_l\delta_{j,k}+h_i\otimes h_k\delta_{j,l}+h_j\otimes h_l\delta_{i,k}+h_j\otimes h_k\delta_{i,l} \rangle_{\mathcal{H}^{\otimes 2}}
\end{align*}
The first term of the previous sum then gives:
\begin{align*}
& \frac{1}{4}\dfrac{\alpha^3}{(n-1)^3}\sum_{i\ne j,\ l\ne k,\ m\ne o}
  \langle h_m\tilde{\otimes} h_o;h_i\otimes
  h_l\delta_{j,k}\rangle_{\mathcal{H}^{\otimes 2}}\\
&=\frac{1}{8}\dfrac{\alpha^3}{(n-1)^3}\sum_{i\ne j,\ l\ne k,\ m\ne o}\langle h_m\otimes h_o+h_o\otimes h_m;h_i\otimes h_l\delta_{j,k}\rangle_{\mathcal{H}^{\otimes 2}},\\
&=\frac{1}{4}\dfrac{\alpha^3}{(n-1)^2}n(n-2).
\end{align*}
The three other terms contribute in a similar way. Thus,
\begin{align*}
(I)=\dfrac{\alpha^3}{(n-1)^2}n(n-2).
\end{align*}
To conclude we need to compute $(II)$. One can check that similar computations lead to:
\begin{align*}
(II)=\dfrac{\alpha^4}{(n-1)^3}n(n-2)(n-3).
\end{align*}
Now, we have:
\begin{align*}
\mid\mid -\alpha f_n+f_n\tilde{\otimes}_{1}f_n\mid\mid^2_{\mathcal{H}^{\otimes 2}}&=\alpha^2\frac{\alpha^2n}{n-1}-2\dfrac{\alpha^4}{(n-1)^2}n(n-2)+\dfrac{\alpha^4}{(n-1)^3}n(n-2)(n-3),\\
&=\alpha^4\frac{n(n-3)}{(n-1)^3}.
\end{align*}
Therefore, for $q=1$, we have the following bound on $\Delta_n$:
\begin{align*}
\Delta_n\leq \alpha^2\sqrt{\frac{n(n-3)}{(n-1)^3}}+\frac{\alpha}{n-1}.
\end{align*}
\noindent
We end this subsection by the following concluding remark.  

\begin{remark}
\begin{itemize}
\item When $d=1$, it is possible to consider in principle more general examples than the previous one. Indeed, consider $\alpha_2,...,\alpha_m$ non-zero real numbers. Consider the following symmetric polynomial kernel in $m\geq 2$ variables:
\begin{align*}
h(x_1,...,x_m)= \alpha_m x_1...x_m+ \alpha_{m-1}\sum_{(i_1,...,i_{m_1}),\ i_k\ne i_l}x_{i_1}...x_{i_{m-1}}+...+\alpha_2\sum_{(i,j),\ i\ne j}x_ix_j.
\end{align*}
Then, we consider the following sequence of random variable ($n\geq m$):
\begin{align*}
U_n(h)=\frac{1}{\binom {n} {m}}\sum_{(i_1,...,i_m),\ i_l\ne i_k}h(Z_{i_1},...,Z_{i_m}),
\end{align*}
with $Z_i$ iid sequence of standard normal random variables. From the theory of $U$-statistics (see \cite{Ser80}), we have the following convergence in law:
\begin{align*}
nU_n(h) \Rightarrow \alpha_2 (Z^2-1),
\end{align*}
where $Z$ is a standard normal random variable. Theorem \ref{bound-secondchaos} provides a quantitative bound for this convergence in distribution in terms of the cumulant of $nU_n(h)$ of order $2$ and $\Gamma_1(nU_n(h))$.
\item When $d=2$, it is possible to perform a similar analysis for the quantity $\Delta_n$ as done in Lemma \ref{ncgamma}. The resulting bounds on $\Delta_n$ would be of the same nature as the ones obtained in the article \cite{ET15} (see in particular Section $5$, Theorem $5.1$ and Corollary $5.10$). To not extend to much the length of the paper, we leave this to the interested readers.
\item When $d\geq 3$, to the best of our knowledge, Theorem \ref{bound-secondchaos} is the first quantitative result for sequences of random variables belonging to a finite sum of Wiener chaoses converging in distribution towards element of the second Wiener chaos. When the sequence $F_n$ lives inside the second Wiener chaos a quantitative bound in Wasserstein-2 distance has been obtained recently in \cite{AAPS}. 
\end{itemize}
\end{remark}



\section{Quantitative bounds for generalized Dickman convergence}
\label{sec:quant-bounds-gener}

\subsection{General bounds for logarithmic combinatorial structures}

Let $0<\theta<\infty$ and $\{X^{\theta}_i\}$ be a sequence of independent integer-valued random variables satisfying the following conditions:
\begin{itemize}
\item We assume that:
\begin{align*}
\underset{i\rightarrow+\infty}{\lim}i\mathbb{P}(X^{\theta}_i=1)=\underset{i\rightarrow+\infty}{\lim}i\mathbb{E}(X^{\theta}_i)=\theta.
\end{align*}
\item For all $i\geq 1$, there exist a random variable $(X^{\theta}_i)^*$ such that:
\begin{align*}
\forall\phi,\ \mathbb{E}[X^{\theta}_i\phi(X^{\theta}_i)]=\mathbb{E}[X^\theta_i]\mathbb{E}[\phi(X^{\theta}_i+(X^{\theta}_i)^*)].
\end{align*}
\item We have:
\begin{align*}
\tilde{\theta}:=\underset{i\geq 1}{\sup}\big(i\mathbb{E}(X^{\theta}_i)\big)<+\infty.
\end{align*}
\item There exists $\gamma>0$ such that:
\begin{align*}
\underset{n\geq 1}{\sup} \sum_{k=1}^n \log\bigg(\mathbb{E}[\exp\big(\frac{k X^{\theta}_k\gamma}{n}\big)]\bigg)<+\infty.
\end{align*}
\end{itemize}
Then, we consider the following sequence of random variables:
\begin{align*}
T_n(\theta)=\frac{1}{n}\sum_{i=1}^n iX^{\theta}_i.
\end{align*}
We introduce the following random variable: let $I$ be a random variable independent of $\{X^{\theta}_i\}$ with values in $\{1,...,n\}$ such that:
\begin{align*}
\mathbb{P}\big(I=j\big)=\dfrac{j\mathbb{E}[X^{\theta}_j]}{\sum_{k=1}^nk\mathbb{E}[X^{\theta}_k]}.
\end{align*}
From the previous assumptions, we have the following properties regarding the sequence $T_n(\theta)$.

\begin{proposition}\label{Tn}
\begin{itemize}
\item $T_n(\theta)$ admits an additive size-bias defined by:
\begin{align*}
\forall\phi,\ \mathbb{E}[T_n(\theta)\phi(T_n(\theta))]=\mathbb{E}[T_n(\theta)]\mathbb{E}[\phi(T_n(\theta)+\frac{I(X^{\theta}_I)^*}{n})].
\end{align*}
\item We have the following convergence in distribution:
\begin{align*}
T_n(\theta)\Rightarrow Z_\theta,
\end{align*}
where $Z_\theta$ is the generalized Dickman distribution defined by:
\begin{align*}
\forall t\in\mathbb{R},\ \mathbb{E}[e^{it Z_\theta}]=\exp\bigg(\theta\int_0^1\dfrac{e^{it x}-1}{x}dx\bigg).
\end{align*}
\item Finally, $Z_\theta$ satisfies the following additive size-bias relation:
\eq
\label{eq:sizebiasdickman}
\mathbb{E}[Z_\theta\phi(Z_\theta)]=\theta\mathbb{E}[\phi(Z_\theta+U)],
\qe
where $U$ is an uniform random variable on $[0,1]$ independent of $Z_\theta$.
\end{itemize}
\end{proposition}

\begin{proof}
Let $\phi$ be a function such that:
\begin{align*}
\mathbb{E}[\mid T_n(\theta)\phi(T_n(\theta))\mid]<+\infty.
\end{align*}
Then, by standard computations, we have:
\begin{align*}
\mathbb{E}[ T_n(\theta)\phi(T_n(\theta))]&=\frac{1}{n}\sum_{i=1}^ni\mathbb{E}[X^{\theta}_i\phi(T^i_n(\theta)+\frac{i X^{\theta}_i}{n})],\\
&=\frac{1}{n}\sum_{i=1}^ni\mathbb{E}[X^\theta_i]\mathbb{E}[\phi(T^i_n(\theta)+\frac{i (X^{\theta}_i+(X^{\theta}_i)^*)}{n})],\\
&=\frac{1}{n}\sum_{i=1}^ni\mathbb{E}[X^\theta_i]\mathbb{E}[\phi(T_n(\theta)+\frac{i (X^{\theta}_i)^*}{n})],\\
&=\mathbb{E}[T_n(\theta)]\mathbb{E}[\phi(T_n(\theta)+\frac{I(X^{\theta}_I)^*}{n})],
\end{align*}
which proves the first bullet point. For the second bullet point, it
is a direct application of Theorem $6.3$ of \cite{ABT03}. For the
third bullet point, this is a consequence of the results contained in
Chapter $4$ of \cite{ABT03} (see in particular relation $(4.22)$
p. $83$). This concludes the proof of the proposition.
\end{proof}
\noindent
Thanks to the approach developed in the previous sections, we obtain the following quantitative bound on the smooth Wasserstein-3 distance between $T_n(\theta)$ and $Z_\theta$. 

\begin{theorem}\label{GeneBound}
Let $\Delta_n$ be the quantity defined by:
\begin{align*}
\Delta_n=\mid \theta -\mathbb{E}[T_n(\theta)] \mid+\mathbb{E}\left[\left| U-\frac{I(X^{\theta}_I)^*}{n}\right|\right].
\end{align*}
Then, we have:
\begin{align*}
\mathcal{W}_3\big(T_n(\theta),Z_\theta\big)\leq C_\theta \Delta_n \sqrt{\mid \log(\Delta_n)\mid},
\end{align*}
for some $C_\theta>0$.
\end{theorem}

\begin{proof}
First of all, we prove the following bound on the difference between the characteristic functions of $T_n(\theta)$ and $Z_\theta$:
\begin{equation}\label{EqDiffChar}
\mid \phi_{T_n(\theta)}(t)-\phi_{Z_\theta}(t)\mid \leq \mid t\mid \mid \theta -\mathbb{E}[T_n(\theta)] \mid+\mid t\mid^2 \mathbb{E}\left[\left| U-\frac{I(X^{\theta}_I)^*}{n}\right|\right].
\end{equation}
Thanks to size bias relations satisfied by $T_n(\theta)$ and $Z_\theta$, we obtain:
\begin{align*}
\forall t>0,\ \phi_{Z_\theta}(t)-\phi_{T_n(\theta)}(t)=i \phi_{Z_\theta}(t)\int_0^t \frac{\psi_n(s)}{\phi_{Z_\theta}(s)}ds,
\end{align*}
with,
\begin{align*}
\psi_n(s)=\mathbb{E}\left[\theta\exp\big(i s (T_n(\theta)+U)\big)-\mathbb{E}[T_n(\theta)]\exp\left(i s (T_n(\theta)+\frac{I(X^\theta_I)^*}{n})\right)\right].
\end{align*}
Moreover, we have a similar relation for $t<0$. Using the fact that:
\eq
\label{eq:phizsphizt}
\forall s\in [0,t],\ \dfrac{1}{|\phi_{Z_\theta}(s)|}\leq \dfrac{1}{|\phi_{Z_\theta}(t)|},
\qe
We obtain the bound:
\begin{equation}\label{diffphiDick}
\mid \phi_{T_n(\theta)}(t)-\phi_{Z_\theta}(t)\mid \leq \mid t\mid \mid \theta -\mathbb{E}[T_n(\theta)] \mid+\mid t\mid^2\mathbb{E}\left[\left| U-\frac{I(X^{\theta}_I)^*}{n}\right|\right].
\end{equation}
Secondly, we note that the last bullet point of the assumptions regarding $\{X^{\theta}_i\}$ implies an uniform sub-exponential concentration bound for $T_n(\theta)$ by Markov inequality. Then, applying Theorem \ref{thm:bounds-char-funct-1}, we obtain the result.
\end{proof}

\subsection{An application towards analysis of algorithms}
\label{sec:an-appl-towards}
In the case $\theta=1$, there is a famous example for which an
explicit rate of convergence is achievable: consider a permutation
$\sigma$ chosen uniformly at random in $\mathfrak{S}_n$ (the set of
permutation of $\{1,...,n\}$). A record of $\sigma$ is a left-to-right
maxima, that is, an element $x_k = \sigma(k)$ in $x_1 \ldots x_n$ is a
record if $x_k\ge x_1, \ldots, x_{k-1}$. The position of this record
is $k$. The src statistic is $ src = \sum_{k=1}^n k \, X_k$ with $X_k$
the random variable which equals 1 if $k$ is a position of a record
and 0 otherwise; this statistic and its asymptotic behavior is
important e.g.\ for analysis of algorithms
\cite{Louchar2014,hwang2002quickselect}.  It is known (see
\cite{renyi1962theorie}) that the random variables
$X_1, X_2, \ldots, X_n$ are independent with $P(X_k=1) =
\frac{1}{k}$. Thus, $\{X_k\}$ clearly satisfy the above
conditions. Now, we have:
\begin{align*}
&\mathbb{E}[T_n]=1,\\
&(X_i)^*=1-X_i.
\end{align*}
Thus, the quantity $\Delta_n$ is equal to:
\begin{align*}
\Delta_n=\mathbb{E}\left[\left| U-\frac{I}{n}+\frac{I X_I}{n}\right|\right].
\end{align*}
Since $\mathbb{E}[\mid \frac{I X_I}{n}\mid]\leq \frac{C_1}{n}$ and since one can construct a coupling $(U, I)$ in such a way that $U$ is independent of $W$ and 
\begin{equation}
  \E \left| U-\frac{I}{n}
     \right| \sim \frac{1}{n}.
\end{equation}
(simply take $U \sim U[0,1]$ and  $I = j$ if $(j-1)/n \le U \le j/n$), we obtain the following bound on $\Delta_n$:
\begin{align*}
\Delta_n\leq \frac{C_2}{n},
\end{align*}
for some $C_2>0$ universal. Thus, Theorem \ref{GeneBound} implies:
\begin{equation}\label{rateDickW3}
\mathcal{W}_3\big(T_n,Z_1\big)\leq \frac{C}{n} \sqrt{ \log(n)}.
\end{equation}
For this simple example, It is possible as well to obtain a relevant quantitative bound 
in Kolmogorov distance using Ess\'een inequality (\ref{eq:1}) together with the fact that
for $\theta \geq 1$, the density of the Dickman distribution is bounded (see Lemma $4.7$ and Corollary $4.8$ of \cite{ABT03}). Set $T:=n^{1/3}$. Then, a direct application of inequalities (\ref{eq:1}) and (\ref{diffphiDick}) leads to:
\begin{equation}\label{rateDickKo}
\operatorname{Kol}\big(T_n,Z_1\big)\leq \frac{C}{n^{\frac{1}{3}}}.
\end{equation}
The previous bounds must be compared with similar results available in
the literature: first of all, our rate of convergence
(\ref{rateDickW3}) is faster than the one which can be obtained from
\cite[Corollary 11.13]{ABT03} in 1-Wasserstein distance. Moreover, one
can compare it as well with \cite[Theorem 3.4]{BN11} where they apply
Stein's method for the compound Poisson to obtain a similar result in
Wasserstein distance (with a rate of $1/n^{\gamma}$ for some
$\gamma>0$). See also Remark \ref{sec:purp-outl-paper}.

\begin{remark}
\begin{itemize}
\item For any $\theta$, we can obtain explicit rates of convergence in the
following simple and important case as well: let $X_i^\theta$ be a
sequence of independent Poisson random variables with parameters
$\theta/i$. Let $T_n(\theta)$ be the corresponding sequence of random
variables under study. Note that:
\begin{align*}
\mathbb{E}[T_n(\theta)]=\theta,\\
(X_i^\theta)^*=1.
\end{align*}
Thus, $\Delta_n$ reduces to:
\begin{align*}
\Delta_n=\mathbb{E}\left[\left| U-\frac{I}{n}\right|\right].
\end{align*}
By a similar coupling argument as in the first bullet point of this remark, we obtain:
\begin{align}
 \mathcal{W}_3(T_n, Z_\theta) \le  C_\theta\dfrac{\sqrt{\log n}}{n}.
\end{align}
\item The fourth assumption concerns concentration properties of the random
variable $T_n(\theta)$ which is the sum of weighted independent random
variables $X^\theta_i$ which admits an additive size biased
version. There is a close connection between coupling of random
variable and its size biased version with concentration properties. In
particular, in this direction, we pinpoint the following recent
article \cite{arratia2013bounded}.  
\end{itemize}
\end{remark}

\subsection{A Dickman limit theorem in number theory}
\label{sec:dickm-limit-theor}
Let $\{p_j\}_{j\geq 1}$ be an enumeration in increasing order of the prime numbers. Let $\{X_j\}$ be a sequence of independent random variables whose laws are defined by:
\begin{align*}
&\mathbb{P}\big(X_j=0\big)= \dfrac{p_j}{1+p_j},\\
&\mathbb{P}\big(X_j=1\big)=\dfrac{1}{1+p_j}.
\end{align*}
Let $S_n$ be defined by:
\begin{align*}
\forall n\geq 1,\ S_n=\dfrac{1}{\log(p_n)}\sum_{j=1}^n X_j \log(p_j).
\end{align*}
It is proved in \cite{CS13} (see Theorem $1$ page $200$) that $S_n$ converges in law to a Dickman distribution ($\theta=1$).\\
\\
Using our approach, we are going to give a new (simple) proof of this convergence result. Moreover, we are going to obtain quantitative bounds on the distance between the distributions of $Z$ and $S_n$; more precisely, we have the following theorem.

\begin{theorem}\label{DickNumTh}
For $n\geq 3$, let $\Delta_n$ be the quantity defined by:
\begin{align*}
\Delta_n=\dfrac{\log\log(n)}{\log(n)}.
\end{align*}
Then there exists $C>0$ such that
\begin{align*}
\mathcal{W}_{3}\big(S_n,Z\big)\leq C\Delta_n\sqrt{-\log(\Delta_n)},
\end{align*}
where $Z$ follows the Dickman distribution ($\theta=1$). In particular,
$$\mathcal{W}_{3}\big(S_n,Z\big) = \mathcal O \left( \frac{(\log \log n)^{3/2}}{\log n} \right).$$
\end{theorem}
\noindent
Before proving Theorem \ref{DickNumTh}, let us introduce some notations. We denote by $Z_n$ the following quantity:
\begin{align*}
Z_n=\frac{1}{\log(p_n)}\sum_{j=1}^n\dfrac{\log(p_j)}{1+p_j}=\mathbb{E}\big[S_n\big]
\end{align*}
Moreover, let $I$ be a random variable independent of $\{X_i\}$ with values in $\{1,...,n\}$, whose law is defined by :
\begin{align*}
\forall j\in\{1,...,n\},\ \mathbb{P}\big(I=j\big)=\dfrac{\log(p_j)}{(p_j+1)\log(p_n)Z_n}.
\end{align*}
The proof is divided into several lemmas. First, we prove that $S_n$ admits a size-biased type random variable denoted $T_n$.
\begin{lemma}\label{SizeBias}
Let $T_n$ be the random variable defined by:
\begin{align*}
T_n=\dfrac{\log(p_I)}{\log(p_n)}-\dfrac{X_I\log(p_I)}{\log(p_n)}.
\end{align*}
Then, for any function $\phi$ such that the following expectations exist, we have:
\begin{align*}
\mathbb{E}\big[S_n\phi(S_n)\big]=Z_n\mathbb{E}\bigg[\phi\big(S_n+T_n\big)\bigg].
\end{align*}
\end{lemma}

\begin{proof}
Let $\phi$ be a function such that:
\begin{align*}
\mathbb{E}\big[S_n\phi(S_n)\big]<+\infty,\ \mathbb{E}\big[\phi\big(S_n+T_n\big)\big]<+\infty.
\end{align*}
Then, we have:
\begin{align*}
\mathbb{E}\big[S_n\phi(S_n)\big]&=\frac{1}{\log\big(p_n\big)}\sum_{j=1}^n\dfrac{\log(p_j)}{p_j+1}\mathbb{E}\left[\phi\left(S_n^j+\dfrac{\log(p_j)}{\log(p_n)}\right)\right]\\
&=Z_n\sum_{j=1}^n\dfrac{\log(p_j)}{(p_j+1)\log(p_n)Z_n}\mathbb{E}\left[\phi\left(S_n^j+\dfrac{\log(p_j)}{\log(p_n)}\right)\right],
\end{align*}
with $S^j_n=S_n-\log(p_j)X_j/\log(p_n)$. By definition of the random variable $I$ (which is independent of $\{X_i\}$), we have:
\begin{align*}
\mathbb{E}\left[S_n\phi(S_n)\right]&=Z_n\mathbb{E}\left[\phi\left(S_n^I+\dfrac{\log(p_I)}{\log(p_n)}\right)\right]\\
&=Z_n\mathbb{E}\left[\phi\left(S_n+\dfrac{\log(p_I)}{\log(p_n)}-\dfrac{X_I\log(p_I)}{\log(p_n)}\right)\right]\\
&=Z_n\mathbb{E}\left[\phi\big(S_n+T_n\big)\right],
\end{align*}
which concludes the proof.
\end{proof}
\noindent
Comparing the previous size-biasing relation to the one of the Dickman distribution \eqref{eq:sizebiasdickman} (here $\theta=1$), we are able to control the closeness between the characteristic functions of $S_n$ and $Z$  in terms of that between $T_n$ and a uniform random variable $U$ and between $Z_n$ and $1$ :

\begin{lemma}
\label{lma:boundphizphisn}
Let $\phi_{S_n}$ (resp. $\phi_Z$) be the characteristic function of $S_n$ (resp. $Z$). Let $U$ be a uniform random variable independent of $S_n$. Then,
\begin{align*}
\mid\phi_{Z}(t)-\phi_{S_n}(t)\mid \leq C\left( |t|\, | 1- Z_n |+|t|^2\, \mathbb{E}| U-T_n|\right).
\end{align*}
\end{lemma}
\begin{proof}
Applying the same procedure as in the proof of Theorem \ref{GeneBound}, combined with Lemma \ref{SizeBias}, we have the following relation for the difference between the characteristic functions of $S_n$ and $Z$,
\begin{align*}
\phi_{Z}(t)-\phi_{S_n}(t)=i \phi_Z(t)\int_0^t \dfrac{\psi_n(s)}{\phi_Z(s)}ds,\\
\psi_n(s)=\mathbb{E}\big[e^{i s S_n}\big(e^{i s U}-\mathbb{E}[S_n]e^{i s T_n}\big)\big].
\end{align*}
Moreover, it holds
\begin{align}\label{eq-bd1}
\mid \psi_n(s) \mid &\leq \mathbb{E}\big[\mid e^{i s U}-\mathbb{E}[S_n]e^{i s T_n}\mid\big],\nonumber \\
&\leq \mid 1- \mathbb{E}[S_n] \mid +\mid s\mid\mathbb{E}\big[\mid U-T_n \mid\big].
\end{align}
This together with \eqref{eq:phizsphizt} implies the result.
\end{proof}

\begin{lemma}
\label{ConvLaw2}
We have 
$$  1- Z_n  = \mathcal O \left( \frac{\log \log n}{\log n} \right),$$
and for a well chosen coupling between $U$ and $T_n$ (with $U$ independent of $S_n$),
$$\mathbb{E}| U-T_n| = \mathcal O \left( \frac{\log \log n}{\log n} \right).$$
\end{lemma}
\begin{proof}

First we write
\begin{align}\label{SecTerm}
\mathbb{E}| U-T_n| \leq \mathbb{E}\left[ \frac{X_I\log(p_I)}{\log(p_n)} \right]+\mathbb{E}\left[\left| U- \frac{\log(p_I)}{\log(p_n)}\right|\right].
\end{align}
Let us prove that for a well chosen coupling $(U,I)$, we have
$$\E\left| U- \frac{\log p_I}{\log p_n}\right| = \mathcal O \left(\frac{\log \log n}{\log n}\right).$$
From the theory of optimal transport, the optimal coupling between $U$ and $\frac{\log p_I}{\log p_n}$ is such that $\frac{\log p_I}{\log p_n}$ is a non-decreasing function of $U$; since $j \mapsto \frac{\log p_j}{\log p_n}$ is itself non-decreasing, it is sufficient to take $I$ as a non-decreasing function of $U$. Hence we define $I$ as follows. We set, for $1\leq j \leq n$,
$$F_j = \sum_{k=1}^j \PP[ I = k] = \frac{1}{Z_n\log p_n}\sum_{k=1}^j \frac{\log p_k}{p_k+1},$$
and $F_0 = 0$. We set $I = j$ if $U \in [F_{j-1},F{j}[$. Note that $U$ depends only on $I$, which is independent of $S_n$; thus $U$ is also independent of $S_n$. If $I=j$, the maximal distance between $U$ and $\frac{\log p_I}{\log p_n}$ is
$$\sup_{u \in [F_{j-1},F{j}[} \left| u - \frac{\log p_j}{\log p_n} \right| = \max\left( \left| F_{j-1} - \frac{\log p_j}{\log p_n} \right|,\left| F_{j} - \frac{\log p_j}{\log p_n} \right| \right).$$
Let us bound $\left| F_{j} - \frac{\log p_j}{\log p_n} \right|$. From Proposition 1.51 of \cite{TV06}, we have $\sum_{k=1}^j \frac{\log p_k}{k} = \log p_j + \mathcal O(1)$, thus
$$\sum_{k=1}^j \frac{\log p_k}{p_k+1} = \log p_j - \sum_{k=1}^j \frac{\log p_k}{(p_k+1)p_k} + \mathcal O(1) = \log j + \mathcal O (\log \log j),$$
using that, by the prime number theorem, $\log p_n = \log n + \mathcal O (\log \log n)$ and $\frac{\log p_k}{(p_k+1)p_k} \sim \frac{1}{k^2 \log k}$. It follows that
$$\sum_{k=1}^j \frac{\log p_k}{p_k+1} - \log p_j = \mathcal O (\log \log j).$$
Note that, since $\log p_n \sim \log n$, this implies in particular that
\eq
\label{eq:bornezn}
Z_n - 1 = \frac{1}{\log p_n} \left( \sum_{k=1}^n \frac{\log p_k}{(p_k+1)}-\log p_n \right) =  \mathcal O (\log \log n/ \log n).
\qe
Thus,
\begin{align*}
F_{j} - \frac{\log p_j}{\log p_n}  &= \frac{1}{\log p_n}  \left( \sum_{k=1}^j \frac{\log p_k}{(p_k+1)} - \log p_j \right) + \frac{1}{\log p_n} \left( Z_n^{-1}-1 \right)\sum_{k=1}^j \frac{\log p_k}{(p_k+1)}\\
&=\frac{\mathcal O (\log \log j)}{\log p_n}  + \frac{1-Z_n}{Z_n \log p_n} (\log j + \mathcal O(\log \log j))\\
&=\mathcal O \left( \frac{\log \log j}{ \log n}\right) + \mathcal O \left(\frac{\log j\log \log n} {\log^2 n}\right).
\end{align*}
Now note that $F_j-F_{j-1} = \PP[I = j] =  \frac{\log p_j}{Z_n\log p_n(1+p_j)} = \mathcal O\left( \frac{1}{j\log n} \right)$, leading to
\begin{align*}
F_{j-1} - \frac{\log p_j}{\log p_n}  &= F_{j-1}-F_j + F_j-  \frac{\log p_j}{\log p_n}\\
&= \frac{\log p_j}{Z_n\log p_n(1+p_j)}+ F_j-  \frac{\log p_j}{\log p_n}\\
&=F_j-  \frac{\log p_j}{\log p_n}+\mathcal O\left( \frac{1}{j\log n} \right).
\end{align*}
Thus, we obtain
$$\sup_{u \in [F_{j-1},F{j}[} \left| u - \frac{\log p_j}{\log p_n} \right| = \mathcal O\left( \frac{\log \log j}{\log n}+ \frac{\log j\log \log n} {\log^2 n} +\frac{1}{j\log n} \right).$$
Using again the fact that $\PP[I = j] = \mathcal O\left( \frac{1}{j\log n} \right)$, we have
\begin{align*}
\E\left| U - \frac{\log p_I}{\log n} \right| &= \sum_{j=1}^n \PP[I=j] \E\left[\left| U - \frac{\log p_I}{\log n}\right|\; \bigg| \; I=j \right]&\\
&=\mathcal O\left(  \sum_{j=1}^n\frac{\log \log j}{j  \log^2 n}+ \frac{\log j\log \log n} {j \log^3 n} +\frac{1}{j^2\log^2 n} \right).
\end{align*}
Finally, it is readily checked that $\sum_{j=1}^n \frac{\log \log j}{j} = \mathcal O(\log n \log\log n)$, $\sum_{j=1}^n \frac{\log j}{j} = \mathcal O(\log^2 n)$, and $\sum_{j=1}^n j^{-2} = \mathcal O(1)$, proving our claim.\\
\\
For the first term on the right hand side of \eqref{SecTerm}, we have:
\begin{align*}
\mathbb{E}\left[ \frac{X_I\log(p_I)}{\log(p_n)} \right]&=\mathbb{E}\left[ \frac{X_I\log(p_I)}{\log(p_n)} \right]\\
&=\sum_{j=1}^{n}\mathbb{P}\big(I=j\big)\mathbb{E}\left[ \frac{X_j\log(p_j)}{\log(p_n)} \right]\\
&=\frac{1}{Z_n}\frac{1}{(\log(p_n))^2}\sum_{j=1}^n\dfrac{\big(\log(p_j)\big)^2}{(1+p_j)^2}.
\end{align*}
Note that,
\begin{align*}
\sum_{j=1}^{+\infty}\dfrac{\big(\log(p_j)\big)^2}{(1+p_j)^2}\leq \sum_{j=1}^{+\infty}\dfrac{\big(\log(j)\big)^2}{(1+j)^2}<\infty.
\end{align*}
Thus,
\begin{align*}
\mathbb{E}\left[ \frac{X_I\log(p_I)}{\log(p_n)} \right]= \mathcal O \left( \frac{1}{\log^2 n} \right).
\end{align*}\\
\\
Finally, the first part of the lemma has been proved in \eqref{eq:bornezn}.
\end{proof}
\noindent
In order to apply the results of Section \ref{sec:bounds-char-funct}, we also need a concentration inequality regarding the random variable $S_n$. We have the following lemma.

\begin{lemma}\label{ConcPrim}
Let $\gamma>0$. Then,
\begin{align*}
\underset{n\geq 1}{\sup}\bigg(\mathbb{E}\big[e^{\gamma S_n}\big]\bigg)<+\infty.
\end{align*}
\end{lemma}

\begin{proof}
Let $\gamma>0$. By standard computations, we have :
\begin{align*}
\log\bigg(\mathbb{E}\big[e^{\gamma S_n}\big]\bigg)&=\sum_{i=1}^n\log\bigg(1+\dfrac{e^{\gamma \frac{\log(p_i)}{\log(p_n)}}-1}{p_i+1}\bigg),\\
&\leq \sum_{i=1}^n \frac{1}{p_i+1}\bigg(e^{\gamma \frac{\log(p_i)}{\log(p_n)}}-1\bigg),\\
&\leq \gamma e^{\gamma} \bigg(\dfrac{1}{\log(p_n)}\sum_{i=1}^n \frac{\log(p_i)}{p_i+1}\bigg),\\
&\leq \gamma e^{\gamma}\; \underset{n\geq 1}{\sup}\bigg(\dfrac{1}{\log(p_n)}\sum_{i=1}^n \frac{\log(p_i)}{p_i+1}\bigg)<+\infty,
\end{align*}
where we have used Proposition 1.51 of \cite{TV06} in the last inequality.
\end{proof}
\noindent
We are now in a position to prove Theorem \ref{DickNumTh}.\\
\\
{\it Proof of Theorem \ref{DickNumTh}.}
From Lemmas \ref{lma:boundphizphisn} and \ref{ConvLaw2}, there exists $C>0$ such that
$$\mid\phi_{Z}(t)-\phi_{S_n}(t)\mid \leq C\; \frac{\log \log n}{\log n} (|t| + |t|^2).$$
From Lemma \ref{ConcPrim} and Theorem \ref{thm:bounds-char-funct-1}, we obtain the result.
\qed
\begin{remark}\label{QuantComp}
Note that the proof we display considerably simplifies the proof of Theorem 1 in \cite{CS13}. Indeed, the authors need to study $12$ sums appearing in the expansion of the characteristic function of $S_n$.
\end{remark}

\section*{Acknowledgements} { YS gratefully acknowledges support from
  the IAP Research Network P7/06 of the Belgian State (Belgian Science
  Policy). YS also thanks Guy Louchard for presenting the problem of
  approximating the distribution of the \emph{src} statistic, as well
  as Giovanni Peccati for enlightening discussions on this topic. The
  research of BA was funded by a Welcome Grant from Universit\'e de
  Li\`ege. The work of GM is funded by FNRS under Grant MIS
  F.4539.16. }

\appendix
\section*{Appendix}
\label{sec:appendix}
\addcontentsline{toc}{section}{Appendix}

\section{Remarks on smooth Wasserstein distance}
\label{sec:remarks-smooth-wass}

Let $n\geq 1$, $0\leq p\leq n$, and
$$B_{p,n} = \{ \phi \in \mathcal C^n(\R) \; | \; ||\phi^{(k)}||_\infty \leq 1, \forall\; p \leq k \leq n \}.$$
For two $\R$-valued random variables $X$ and $Y$, we define the following distances between their distributions :
\begin{align*}
&d^{p,n}_W(X,Y) = \sup_{\phi \in B_{p,n}} \E[\phi(X) - \phi(Y)],\\
&d^{p,n}_\infty(X,Y) = \sup_{\phi \in B_{p,n} \cap \mathcal C^\infty(\R)} \E[\phi(X) - \phi(Y)],\\
&d^{p,n}_c(X,Y) = \sup_{\phi \in B_{p,n} \cap \mathcal C_c^\infty(\R)} \E[\phi(X) - \phi(Y)],
\end{align*}
where $\mathcal C_c^\infty(\R)$ the space of smooth functions with compact support. We also note
\eq
\label{eq:defwn}
\mathcal W_n (X,Y) = d_{W}^{0,n} (X,Y).
\qe
\begin{lemma}
\label{lem:app1}
Assume $X$ and $Y$ admit a moment of order $p$. Then
$$d^{p,n}_W(X,Y) = d^{p,n}_\infty(X,Y).$$
\end{lemma}
\begin{proof}
It is clear that $d^{p,n}_\infty(X,Y) \leq d^{p,n}_W(X,Y).$

Assume first that $\E[X^k] > \E[Y^k]$ for some $k\leq p-1$. Let $\phi_A(x) = A x^k$ with $A>0$. Then $\phi \in B_{p,n} \cap \mathcal C^\infty(\R)$, so $d^{p,n}_\infty(X,Y) \geq A (\E[X^k] - \E[Y^k])$, leading to $d^{p,n}_\infty(X,Y) = d^{p,n}_W(X,Y) = +\infty$. From now on, we assume that $\E[X^k] = \E[Y^k]$ for every $0\leq p \leq k-1$.

Let $\omega$ be the pdf of the standard normal distribution, $0<\epsilon <1$ and $\omega_\epsilon(x) = \epsilon^{-1} \omega(\epsilon^{-1} x)$. Let $\phi \in B_{p,n}$ and define
$$\phi_\epsilon(x) = \int \phi(y) \omega_{\epsilon}(x-y) dy.$$
Since for every $ k \leq p-1$, the moments of order $k$ of $X$ and $Y$ match, on can assume without loss of generality that $\phi(0) = \ldots = \phi^{(p-1)} (0) = 0$. Since $||\phi^{(p)}||_\infty \leq 1$, then
\begin{equation}
\label{eq:12345}
\forall \; 0\leq k \leq p, |\phi^{(k)}(x) | \leq C |x|^{p-k},
\end{equation}
where $C$ is a constant not depending on $\phi$. From \eqref{eq:12345} applied to $k=0$, et by Lebesgue's dominated convergence theorem, we have that $\phi_\epsilon \in \mathcal C^\infty(\R)$, and for all $k \leq n$, $\phi_\epsilon^{(k)}(x) = \int \phi^{(k)}(x-y) \omega_{\epsilon}(y) dy$. Thus $\phi_\epsilon^{(k)} \in B_{p,n} \cap  \mathcal C^\infty(\R)$. We deduce
\begin{align}
\E[\phi(X) - \phi(Y)] &= \E[\phi_\epsilon(X) - \phi_\epsilon(Y)] +   \E[\phi(X) - \phi_\epsilon(X)]+ \E[\phi_\epsilon(Y) - \phi(Y)]\notag \\
&\leq  d_\infty^{p,n}(X,Y)+   \E[\phi_\epsilon(X) - \phi(X)]+ \E[\phi_\epsilon(Y) - \phi(Y)]. \label{eq:2}
\end{align}
Now, note that $ \E[\phi_\epsilon(X)] = \E[\phi(X+\epsilon Z)]$, where $Z$ is a standard normal r.v. independent of $X$. Then we have
\begin{align*}
|\E[\phi(X+\epsilon Z) - \phi(X)]| &\leq \epsilon \sum_{k=1}^{p-1}\epsilon^{k-1} \E[\frac{|\phi^{(k)}(X)|}{k!}|Z|^k] + \epsilon^{p} \E[\frac{|\phi^{(p)}(\tilde X)|}{p!}|Z|^p]\\
&\leq  C \epsilon \sum_{k=1}^{p-1} \E[|X|^{p-k}]\E[|Z|^k] + C\epsilon^{p} \E[|Z|^p]\\
&\leq C \epsilon,
\end{align*}
where $\tilde X \in  [X,X+\epsilon Z]$ and $C$ does not depend on $\phi$ nor on $\epsilon$. Apply this inequality to \eqref{eq:2}, take the supremum over $\phi$ and let $\epsilon$ go to $0$ to achieve the proof.
\end{proof}

We now state a technical Lemma.
\begin{lemma}
\label{lem:2}
For every $n \geq 1$, there exists $C>0$ and $\epsilon_0>0$ such that, for all $\epsilon\in (0,\epsilon_0)$ and $M \geq 1$, there exists $f \in \mathcal C_c^\infty(\R)$ with values in $[0,1]$ verifying the following :
$$f(x) = 1, \; \forall x \in [-M,M],$$
and
  $$|f^{(k)}(x)| \leq \frac{ C \epsilon}{|x|}, \quad\forall x \in \R, \forall k \in [|1,n|].$$
\end{lemma}
\begin{proof}
Let $g \in \mathcal C^\infty(\R)$ with values in $[0,1]$ such that $g(x)= 0$ for all $x\leq 0$ and $g(x) = 1$ for all $x \geq 1$. We will first define the derivative of the function $f$, then integrate. Let $\epsilon>0$, $M\geq 1$ and $A\geq M+1$. We define $f_{A,\epsilon}$ as follows :
$$
f_{A,\epsilon}(x) = \left\{
\begin{array}{ll}
\frac{\epsilon g(x-M)}{x} &\textrm{if } x \leq M+1 \\
\frac{\epsilon g(A+1-x)}{x} &\textrm{if } x \geq A \\
\frac{\epsilon }{x} &\textrm{if } x \in [M+1,A]. \\
\end{array}
\right.
$$
Let $h(A) = \int_{\R} f_{A,\epsilon} (x) dx$. Then $h(M+1) \leq 2\epsilon ||g||_\infty$, so that if $\epsilon_0<\frac{1}{2||g||_\infty}$, we have that for all $\epsilon<\epsilon_0$, $h(M+1)<1$. On the other hand, $h(A) \rightarrow +\infty$ when $A\rightarrow +\infty$, and since $h$ is continuous one can choose $A$ such that $h(A) = 1$. Then we define for every $x\in \R$,
$$f(x) = 1-\int_0^{|x|} f_{A,\epsilon} (t) dt.$$
Let us prove that $f$ has the required properties. It is clear that $f \in \mathcal C_c^\infty(\R)$ and that $f(x) =1$ if $x \in [-M,M]$. If $1\leq k \leq n$ and $A \leq x \leq A+1$,
\begin{align*}
|f^{(k)}(x)| &\leq \epsilon \sum_{l=0}^{k-1} C_{k-1}^l l! \frac{||g^{(k-1-l)}||_\infty}{|x|^{1+l}} \\
&\leq \frac{\epsilon C}{|x|},
\end{align*}
where $C = n! 2^n \max_{0\leq l \leq n} ||g^{(l)}||_\infty$. The same bound holds on $[M,M+1]$. On $[M+1,A]$, we have $|f^{(k)}(x)| \leq \epsilon n! / |x|$. Since $f$ is even, the same bounds hold for $x \leq 0$.
\end{proof}

\begin{lemma}
Assume $p=0$ or $p=1$, and $n \geq 1$. Then, for every random variables $X,Y$ with finite moment of order $p$,
$$d^{p,n}_W(X,Y) = d^{p,n}_c(X,Y).$$
\end{lemma}
\begin{proof}
From Lemma \ref{lem:app1}, it suffices to show that $d^{p,n}_\infty=d^{p,n}_c$.

It is clear that $d_c \leq d_\infty$.

Let $\epsilon >0$. Let $M>1$ such that if $K = [-M,M]$,
$$\E[ |X|^p\; \mathbf 1_{X \in K^c}] \leq \epsilon,$$
$$\E[ |Y|^p\; \mathbf 1_{Y \in K^c}] \leq \epsilon.$$

Let $\phi \in B_{p,n} \cap \mathcal C^\infty(\R)$. We may assume that $\phi(0) = 0$, in which case, since $||\phi'||_\infty\leq 1$, $|\phi(x)| \leq |x|$.

Let $f$ be the function defined in Lemma \ref{lem:2} and $\tilde \phi = \phi f$. In the following, $C$ is a constant independent of $\epsilon$ and $\phi$ which may vary from line to line. Then for all $k \in [|1,n|]$ :
\begin{align*}
|\tilde \phi^{(k)}(x)| &=  \bigg|\sum_{l=0}^k C_k^l f^{(l)}(x) \phi^{(k-l)}(x)\bigg|\\
&= \bigg| f(x) \phi^{(k)}(x) + \sum_{l=1}^{k-p} C_k^l f^{(l)}(x) \phi^{(k-l)}(x) +  \sum_{l=k-p+1}^{k} C_k^l f^{(l)}(x) \phi^{(k-l)}(x) \bigg|\\
&\leq 1 + C \epsilon +   |f^{(k)}(x)| |\phi(x)|\\
&\leq 1 + C \epsilon +   \frac{C\epsilon}{|x|}\, |x|\\
&\leq 1 + C \epsilon
\end{align*}
Thus, $\tilde \phi / (1+C \epsilon) \in B_n \cap \mathcal C_c^\infty(\R)$. This leads to
\begin{align*}
\E[\phi(X) - \phi(Y)] &= \E[\tilde \phi(X) - \tilde \phi(Y)] + \E[ \phi(X)(1-f(X))] - \E[\phi(Y)(1-f(Y))]\\
& \leq (1+C \epsilon) d_c(X,Y) + 2 \epsilon.
\end{align*}
By taking the supremum over $\phi$ and letting $\epsilon$ go to zero, we obtain $d_\infty \leq d_c$, which achieves the proof.
\end{proof}

\section{Stable distributions}
\label{sec:stable-distributions}

Let $\alpha\in (1,2)$. Let $c=(1-\alpha)/(2\Gamma(2-\alpha)\cos(\pi\alpha/2))$. We denote by $Z^{\alpha}$ a symmetric $\alpha$-stable random variable whose characteristic function is given by:
\begin{align*}
\forall\xi\in\mathbb{R},\ \phi_{Z^{\alpha}}(\xi)=\exp\bigg(-\mid\xi\mid^\alpha\bigg).
\end{align*}
We define the following differential operator on smooth enough function:
\begin{align*}
\mathcal{D}^{\alpha-1}(\psi)(x)=\frac{1}{2\pi}\int_{\mathbb{R}}\mathcal{F}(\psi)(\xi)\big(\dfrac{i\alpha\mid\xi\mid^\alpha}{\xi}\big)e^{i\xi x}d\xi
\end{align*}
We have the following straightforward Stein-type characterization identity:
\begin{align*}
\mathbb{E}\big[Z^{\alpha}\psi(Z^{\alpha})\big]=\mathbb{E}[\mathcal{D}^{\alpha-1}(\psi)(Z^{\alpha})].
\end{align*}
We denote by $\operatorname{Dom}(Z^{\alpha})$ the normal domain of
attraction of $Z^{\alpha}$. We recall the following result (see
Theorem $5$ page $81$ of \cite{GK68} and Theorem $2.6.7$ page $92$ of
\cite{IL65}).
\begin{theorem}\label{StableCLT}
A distribution function, $F$, is in the normal domain of attraction of $Z^{\alpha}$ if and only if:
\begin{align*}
&\forall x>0,\ F(x)=1-\dfrac{(c+a_1(x))}{x^{\alpha}},\\
&\forall x<0,\ F(x)=\dfrac{(c+a_2(x))}{(-x)^{\alpha}},
\end{align*}
with $\underset{x\rightarrow+\infty}{\lim}a_1(x)=\underset{x\rightarrow-\infty}{\lim}a_2(x)=0$.
\end{theorem}
\begin{remark}\label{examples}
Let $\lambda=(2c)^{1/\alpha}$ and consider the following Pareto type distribution function:
\begin{align*}
&\forall x>0,\ F_\lambda(x)=1-\dfrac{1}{2(1+\frac{x}{\lambda})^{\alpha}},\\
&\forall x<0,\ F_\lambda(x)=\dfrac{1}{2(1-\frac{x}{\lambda})^{\alpha}}.
\end{align*}
It is easy to check that $F_\lambda$ is in $\operatorname{Dom}(Z^{\alpha})$. In particular, we have:
\begin{align*}
\forall x>0,\ a_1(x)=a_2(-x)=\dfrac{x^\alpha}{2\big(1+\frac{x}{\lambda}\big)^\alpha}-c.
\end{align*}
\end{remark}
\noindent
Let $X$ be a random variable in $\operatorname{Dom}(Z^{\alpha})$ such that $\mathbb{E}[X]=0$. We make the following assumptions:
\begin{itemize}
\item The functions $a_1(.)$ and $a_2(.)$ are continuous and bounded on $\mathbb{R}^*_+$ and on $\mathbb{R}^*_-$ respectively. Moreover, they satisfy:
\begin{align*}
&\underset{x\rightarrow+\infty}{\lim}xa_1(x)<+\infty,\ \underset{x\rightarrow-\infty}{\lim}xa_2(x)<+\infty.
\end{align*}
\end{itemize}
Let $(X_i)$ be independent copies of $X$. We define the random variable $W$ by:
\begin{align*}
W=\frac{1}{n^{\frac{1}{\alpha}}}\sum_{i=1}^nX_i.
\end{align*}
We consider the following function on $\mathbb{R}^*$:
\begin{align*}
\phi_X^*(\xi)=\dfrac{-\xi}{\alpha\mid\xi\mid^\alpha}\mathbb{E}\big[iX\exp\big(i\xi X\big)\big].
\end{align*}
Note in particular that this function is well-defined, continuous on
$\mathbb{R}^*$ and
$\underset{\mid\xi\mid\rightarrow +\infty}{\lim}\phi_X^*(\xi)=0$ since
$X\in \operatorname{Dom}(Z^{\alpha})$ and $\alpha\in (1,2)$. 

\begin{proposition}\label{properties}
The function $\phi_X^*$ satisfies:
\begin{itemize}
\item $\phi_X^*(0^+)=\phi_X^*(0^-)=1$.
\item If the law of $X$ is symmetric, then $\phi_X^*$ is real-valued and even.
\item There exists a tempered distribution $T_X^*$ such that $\overline{\phi_X^*}$ is the Fourier transform of $T_X^*$.
\item For all $\psi$ smooth enough, we have:
\begin{align*}
\mathbb{E}\big[X\psi(X)\big]=<T_X^*;\mathcal{D}^{\alpha-1}(\psi)>
\end{align*}
\item We have the following formulae:
\begin{align*}
&\forall\xi>0,\ \phi_X(\xi)- \phi_{Z^{\alpha}}(\xi)=\alpha e^{-\xi^\alpha}\int_0^\xi\nu^{\alpha-1}e^{\nu^\alpha}\big(\phi_X(\nu)-\phi^*_X(\nu)\big)d\nu,\\
&\forall\xi<0,\ \phi_X(\xi)- \phi_{Z^{\alpha}}(\xi)=\alpha e^{-(-\xi)^\alpha}\int_\xi^0(-\nu)^{\alpha-1}e^{(-\nu)^\alpha}\big(\phi_X(\nu)-\phi^*_X(\nu)\big)d\nu.
\end{align*}
\end{itemize}
\end{proposition}
\begin{proof}
See the appendix \ref{sec:technical-proofs}.
\end{proof}

\begin{remark}
The functional equality appearing in the fourth bullet point of the previous proposition is very close to the definition of the zero-bias transformation distribution for centred random variable with finite variance (see \cite{GR97}). It would be interesting to know if $T_X^*$ is actually a linear functional defined by a probability measure on $\mathbb{R}$.
\end{remark}
\noindent
Since the $X_i$ are independent and identically distributed, we have the following formula for the characteristic function of $W$:
\begin{align*}
\forall\xi\in\mathbb{R},\ \phi_W(\xi)=\phi_X^n\big(\frac{\xi}{n^{\frac{1}{\alpha}}}\big)
\end{align*}
Consequently, this characteristic function is differentiable and we have:
\begin{align*}
\phi^*_W(\xi)=\phi_X^{n-1}\big(\frac{\xi}{n^{\frac{1}{\alpha}}}\big)\phi_X^*\big(\frac{\xi}{n^{\frac{1}{\alpha}}}\big).
\end{align*}
Thus, we can apply the previous proposition to $W$ and we get the following formula for $\xi>0$:
\begin{align}
\phi_W(\xi)-\phi_{Z^{\alpha}}(\xi)&=\alpha e^{-\xi^\alpha}\int_0^\xi\nu^{\alpha-1}e^{\nu^\alpha}\big(\phi_W(\nu)-\phi^*_W(\nu)\big)d\nu, \nonumber\\
&=\alpha e^{-\xi^\alpha}\int_0^\xi\nu^{\alpha-1}e^{\nu^\alpha}\phi_X^{n-1}\big(\frac{\nu}{n^{\frac{1}{\alpha}}}\big)\big(\phi_X\big(\frac{\nu}{n^{\frac{1}{\alpha}}}\big)-\phi_X^*\big(\frac{\nu}{n^{\frac{1}{\alpha}}}\big)\big)d\nu.\label{eq-formulaW1}
\end{align}
From the previous formula, we deduce the following straightforward pointwise bound:
\begin{align*}
|\phi_W(\xi)-\phi_{Z^{\alpha}}(\xi)| \leq \alpha e^{-\xi^\alpha}\int_0^\xi\nu^{\alpha-1}e^{\nu^\alpha}\mid\big(\phi_X\big(\frac{\nu}{n^{\frac{1}{\alpha}}}\big)-\phi_X^*\big(\frac{\nu}{n^{\frac{1}{\alpha}}}\big)\big)\mid d\nu. 
\end{align*}
Now, the objective is twofold:
\begin{itemize}
\item Find an efficient way to bound the term  $\nu^{\alpha-1}\mid\big(\phi_X\big(\frac{\nu}{n^{\frac{1}{\alpha}}}\big)-\phi_X^*\big(\frac{\nu}{n^{\frac{1}{\alpha}}}\big)\big)\mid$.
\item Bound finely the Dawson-type function associated with the characteristic function of the stable law.
\end{itemize}
To deal with the second bullet point, we have the following lemma.

\begin{lemma}\label{DawsonStable}
For any $\alpha\in (1,2)$, there exists a positive constant $C_1(\alpha)>0$ such that, for $\xi>0$:
\begin{align*}
e^{-\xi^\alpha}\int_0^\xi e^{\nu^\alpha}d\nu\leq \left\{
    \begin{array}{ll}
       C_1(\alpha)e^{-\xi^\alpha}+\xi^{1-\alpha}  & \xi>1, \\
       C_1(\alpha)e^{-\xi^\alpha} & \xi<1.
    \end{array}
\right.
\end{align*}
\end{lemma}

\begin{proof}
Let $\xi\geq 1$. We have by integration by parts:
\begin{align*}
\int_1^\xi e^{\nu^\alpha}d\nu&=\int_1^\xi \nu^{1-\alpha} \nu^{\alpha-1} e^{\nu^\alpha}d\nu,\\
&=\big[\nu^{1-\alpha}\frac{e^{\nu^\alpha}}{\alpha}\big]_1^\xi+\frac{\alpha-1}{\alpha}\int_1^\xi\nu^{-\alpha}e^{\nu^\alpha}.
\end{align*}
Roughly bounding the previous formula, we obtain:
\begin{align*}
\int_1^\xi e^{\nu^\alpha}d\nu\leq \frac{1}{\alpha}\xi^{1-\alpha}e^{\xi^\alpha}+\frac{\alpha-1}{\alpha}\int_1^\xi x^{-\alpha}e^{x^\alpha}dx,
\end{align*}
which implies:
\begin{align*}
\int_1^\xi e^{\nu^\alpha}d\nu\leq \xi^{1-\alpha}e^{\xi^\alpha}.
\end{align*}
This concludes the proof of the lemma.
\end{proof}
\noindent
To deal with the first bullet point, we use the following formula first obtained in the proof of Proposition \ref{properties}:
\begin{align}
\phi_X^*(\xi)&=1+\frac{1}{\alpha}\int_0^{+\infty}e^{ix}\frac{a_1\big(\frac{x}{\xi}\big)}{x^{\alpha-1}}dx\nonumber\\
&-\frac{i}{\alpha}\int_0^{+\infty}(e^{i x}-1)\frac{a_1\big(\frac{x}{\xi}\big)}{x^\alpha}dx+\frac{1}{\alpha}\int_{-\infty}^0e^{ix}\frac{a_2\big(\frac{x}{\xi}\big)}{(-x)^{\alpha-1}}dx\nonumber\\
&+\frac{i}{\alpha}\int_{-\infty}^0(e^{i x}-1)\frac{a_2\big(\frac{x}{\xi}\big)}{(-x)^\alpha}dx.\label{eq-formulaW2}
\end{align}
Thus, the difference between $\phi_X\big(\xi\big)$ and $\phi_X^*\big(\xi\big)$ brings into play the difference between $\phi_X\big(\xi\big)$ and $1$. Since $X\in \operatorname{Dom}(Z^{\alpha})$, we expect this difference to be of order $\mid\xi\mid^{\alpha}$. This is the aim of the next lemma.
\begin{lemma}\label{orderN}
There exists some strictly positive constant, $C_{\alpha}$, depending on $\alpha$ only, such that, for any $\xi$ in $\mathbb{R}$, we have:
\begin{align*}
\mid\phi_X\big(\xi\big)-1\mid\leq C_{\alpha}\mid\xi\mid^{\alpha}\big(1+\mid\mid a_1\mid\mid_{\infty}+\mid\mid a_2\mid\mid_{\infty}\big)
\end{align*} 
\end{lemma}
\begin{proof}
See the appendix \ref{sec:technical-proofs}.
\end{proof}

\begin{remark}
We note in particular that the previous lemma implies:
\begin{align*}
\left|\phi_X\big(\frac{\xi}{n^{\frac{1}{\alpha}}}\big)-1\right|\leq C_{\alpha}\dfrac{\mid\xi\mid^{\alpha}}{n}\big(1+\mid\mid a_1\mid\mid_{\infty}+\mid\mid a_2\mid\mid_{\infty}\big).
\end{align*}
\end{remark}
\noindent
Now, we want to bound pointwisely the residual term, namely:
\begin{align*}
R^{\alpha}(\xi)&=\frac{1}{\alpha}\int_0^{+\infty}e^{ix}\frac{a_1\big(\frac{x}{\xi}\big)}{x^{\alpha-1}}dx-\frac{i}{\alpha}\int_0^{+\infty}(e^{i x}-1)\frac{a_1\big(\frac{x}{\xi}\big)}{x^\alpha}dx\\
&+\frac{1}{\alpha}\int_{-\infty}^0e^{ix}\frac{a_2\big(\frac{x}{\xi}\big)}{(-x)^{\alpha-1}}dx+\frac{i}{\alpha}\int_{-\infty}^0(e^{i x}-1)\frac{a_2\big(\frac{x}{\xi}\big)}{(-x)^\alpha}dx.
\end{align*}
For this purpose, we have the following lemma.
\begin{lemma}\label{pointwise}
For any $i\in\{1,2\}$, there exist strictly positive constants $C^i_{\alpha,1}$ and $C^i_{\alpha,2}$ depending on the function $a_i$ and $\alpha$ only, such that:
\begin{align*}
&\forall \xi\in\mathbb{R},\ \bigg\lvert\int_0^{+\infty}(e^{i x}-1)\frac{a_1\big(\frac{x}{\xi}\big)}{x^\alpha}dx\bigg\rvert \leq C^1_{\alpha,1}\mid \xi\mid^{\frac{2-\alpha}{2}},\\
&\forall \xi\in\mathbb{R},\ \bigg\lvert\int_0^{+\infty}e^{ix}\frac{a_1\big(\frac{x}{\xi}\big)}{x^{\alpha-1}}dx\bigg\rvert \leq C^1_{\alpha,2}\mid \xi\mid^{2-\alpha},
\end{align*}
and similarly for the integrals with the function $a_2(.)$.
\end{lemma}
\begin{proof}
See the appendix \ref{sec:technical-proofs}.
\end{proof}
\noindent
From the two previous lemmas, we are now in position to provide a bound on the following term for any $\nu\in [0,\xi]$:
\begin{equation}\label{pointwiseN1}
\nu^{\alpha-1}\mid\big(\phi_X\big(\frac{\nu}{n^{\frac{1}{\alpha}}}\big)-\phi_X^*\big(\frac{\nu}{n^{\frac{1}{\alpha}}}\big)\big)\mid\leq C_\alpha \bigg( \frac{\xi^{2\alpha-1}}{n}+\frac{\xi^{\frac{\alpha}{2}}}{n^{\frac{1}{\alpha}-\frac{1}{2}}}+\frac{ \xi}{n^{\frac{2}{\alpha}-1}}\bigg).
\end{equation}
\noindent
Combining the bound (\ref{pointwiseN1}) together with Lemma \ref{DawsonStable}, we obtain the following bound for the difference between the characteristic function of $W$ and the one of $Z^{\alpha}$ for $\xi\in\mathbb{R}$:
\begin{equation}\label{EstStableDiff}
|\phi_W(\xi)-\phi_{Z^{\alpha}}(\xi)| \leq C_\alpha \bigg( \frac{\mid\xi\mid^{\alpha}}{n}+\frac{\mid\xi\mid^{1-\frac{\alpha}{2}}}{n^{\frac{1}{\alpha}-\frac{1}{2}}}+\frac{\mid \xi\mid^{2-\alpha}}{n^{\frac{2}{\alpha}-1}}\bigg).
\end{equation}
We conclude the discussion on stable convergence by applying a
variant of Theorem~\ref{theor:bounds-char-funct-1} to obtain
correct-order rates of convergence towards stable distributions in smooth Wasserstein distance. We use as well the bound \ref{EstStableDiff} together with Ess\'een inequality (\ref{eq:1}) to provide a rate of convergence in Kolmogorov distance.
\begin{theorem}
  \label{theor:ratesstable-distributions}
There exists a constant $C>0$ depending on $\alpha, a_1$ and $a_2$
such that 
\begin{equation}
  \label{eq:41}
  \mathcal{W}_3(W, Z^{\alpha}) \le C \,  {n^{\frac{2\alpha}{2\alpha+1}\left(\frac{1}{2}-\frac{1}{\alpha}\right)} }.
\end{equation}
\end{theorem}
\begin{remark}
This result should be compared with the one of \cite{JS05} (see also \cite{rachev1994rate}). In particular, they obtain bounds of the order $n^{1/2 - 1/\alpha}$ for the classical 2-Wasserstein distance.
\end{remark}

\begin{theorem}\label{kolmostable}
There exists a constant strictly positive $C>0$ depending on $\alpha$, $a_1$ and $a_2$ such that:
\begin{align*}
\underset{x\in\mathbb{R}}{\sup}\mid \mathbb{P}\big(W\leq x\big)- \mathbb{P}\big(Z^{\alpha}\leq x\big)\mid \leq \frac{C}{n^{\frac{1}{2}(1-\frac{\alpha}{2})}}.
\end{align*}
\end{theorem}

\begin{proof}
By Esséen smoothing lemma, for any $k\geq 1$ and for any $T>0$, we have:
\begin{align*}
\mid \mathbb{P}\big(W\leq x\big)- \mathbb{P}\big(Z^{\alpha}\leq x\big)\mid \leq \frac{k}{2\pi}\int_{-T}^T\mid \frac{\phi_W(t)-\phi_{Z^{\alpha}}(t)}{t}\mid dt+\frac{c(k)}{T}.
\end{align*}
Plugging inequality (\ref{EstStableDiff}) and integrating with respect to $t$, we get:
\begin{align*}
\mid \mathbb{P}\big(W\leq x\big)- \mathbb{P}\big(Z^{\alpha}\leq x\big)\mid \leq C_{k,\alpha}\bigg(\frac{T^\alpha}{n}+\frac{T^{1-\frac{\alpha}{2}}}{n^{\frac{1}{\alpha}-\frac{1}{2}}}+\frac{T^{2-\alpha}}{n^{\frac{2}{\alpha}-1}}\bigg)+\frac{c(k)}{T}.
\end{align*}
Now, we set $T:= n^{1/\alpha-1/2}$. By straightforward computations, we get:
\begin{align*}
\mid \mathbb{P}\big(W\leq x\big)- \mathbb{P}\big(Z^{\alpha}\leq x\big)\mid \leq C_{k,\alpha}\bigg( \frac{1}{n^{\frac{\alpha}{2}}}+\frac{1}{n^{\frac{1}{2}(1-\frac{\alpha}{2})}}+\frac{1}{n^{1-\frac{\alpha}{2}}}\bigg)+\frac{c(k)}{n^{\frac{1}{\alpha}-\frac{1}{2}}}
\end{align*}
Now we note that for any $\alpha\in (1,2)$ we have the following inequalities:
\begin{align*}
&\frac{1}{2}(1-\frac{\alpha}{2})\leq (1-\frac{\alpha}{2})\leq \frac{\alpha}{2},\\
&\frac{1}{2}(1-\frac{\alpha}{2})\leq \frac{1}{\alpha}-\frac{1}{2}.
\end{align*}
This concludes the proof of the theorem.
\end{proof}

\begin{remark}\label{CompStable}
\begin{itemize}
\item From the results contained in \cite{JS05}, one can obtain a rate of convergence in Kolmogorov distance using the fact that the Kolmogorov distance can be bounded by the square root of the Wasserstein distance when the limiting density is bounded (which is clearly the case). Thus, using the results of \cite{JS05}, one would have the following rate of convergence:
\begin{align*}
\underset{x\in\mathbb{R}}{\sup}\mid \mathbb{P}\big(W\leq x\big)- \mathbb{P}\big(Z^{\alpha}\leq x\big)\mid \leq \dfrac{C_\alpha}{n^{\frac{1}{2}(\frac{1}{\alpha}-\frac{1}{2})}}.
\end{align*}
We note that:
\begin{align*}
\frac{1}{2}(1-\frac{\alpha}{2})-\frac{1}{2}(\frac{1}{\alpha}-\frac{1}{2})\geq 0 \Leftrightarrow (2-\alpha)(\alpha-1)\geq 0.
\end{align*}
which is always the case since $\alpha\in (1,2)$.
\item Under the assumptions on the functions $a_1$ and $a_2$, we can obtain a rate of convergence in Kolmogorov distance thanks to Theorem $1$ of \cite{Hall81}. Indeed, we note that the functions $S(.)$ and $D(.)$ of this article admit the following expressions:
\begin{align*}
\forall x>0,\ &S(x):=\dfrac{a_1(x)+a_2(-x)}{x^\alpha},\\
&D(x):=\dfrac{a_1(x)-a_2(-x)}{x^\alpha}
\end{align*}
Since $\underset{x\rightarrow +\infty}{\lim} xa_1(x)<+\infty$ and $\underset{x\rightarrow -\infty}{\lim} xa_2(x)<+\infty$, for any $\epsilon\in (0, 2-\alpha)$, we have that:
\begin{align*}
\mid S(x) \mid+\mid D(x) \mid = o(x^{-(\alpha+\epsilon)}),
\end{align*}
which produces the rate:
\begin{align*}
\underset{x\in\mathbb{R}}{\sup}\mid \mathbb{P}\big(W\leq x\big)- \mathbb{P}\big(Z^{\alpha}\leq x\big)\mid \leq C_\alpha \dfrac{1}{n^{\frac{\epsilon}{\alpha}}}.
\end{align*}
In particular, we can find $\epsilon\in (0, 2-\alpha)$ such that $\frac{\epsilon}{\alpha}=\frac{1}{2}(1-\frac{\alpha}{2})$.
\end{itemize}
\end{remark}

\subsection{Technical proofs}
\label{sec:technical-proofs}

\begin{proof}[Proof of Proposition \ref{properties}]
Let us prove the first bullet point. We denote by $F_X$ the distribution function of $X$. Let $\xi>0$. By definition and using the fact that $\mathbb{E}[X]=0$, we have:
\begin{align*}
\mathbb{E}\big[iX\exp\big(i\xi X\big)\big]&=i\int_{\mathbb{R}}xe^{i\xi x}dF_X(x),\\
&=i\int_{\mathbb{R}}x(e^{i\xi x}-1)dF_X(x),\\
&=i\int_0^{+\infty}x(e^{i\xi x}-1)dF_X(x)+i\int_{-\infty}^{0}x(e^{i\xi x}-1)dF_X(x)
\end{align*}
Let us deal with the first integral. For this purpose, we fix some interval $[a,b]$ strictly contained in $(0,+\infty)$. We note that $F_X$ is of bounded variation on $[a,b]$ as well as the function $x\rightarrow c/x^{\alpha}$. Thus, the function $x\rightarrow a_1(x)/x^\alpha$ is of bounded variation on $[a,b]$. We can write:
\begin{align*}
\int_a^{b}x(e^{i\xi x}-1)dF_X(x)&=\alpha c\int_a^{b}(e^{i\xi x}-1)\frac{dx}{x^\alpha}-\int_{a}^bx(e^{i\xi x}-1)d\big(\frac{a_1(x)}{x^\alpha}\big),\\
&=\alpha c\int_a^{b}(e^{i\xi x}-1)\frac{dx}{x^\alpha}+i\xi\int_{a}^bxe^{i\xi x}\frac{a_1(x)}{x^\alpha}dx+\int_{a}^b(e^{i\xi x}-1)\frac{a_1(x)}{x^\alpha}dx\\
&-\bigg[x(e^{i\xi x}-1)\frac{a_1(x)}{x^\alpha}\bigg]_a^b,
\end{align*}
where we have performed an integration by parts on the last line. Using the assumptions upon the function $a_1(.)$, we can let $a$ and $b$ tend to $0^+$ and $+\infty$ respectively in order to obtain:
\begin{align*}
\int_0^{+\infty}x(e^{i\xi x}-1)dF_X(x)&=\alpha c\int_0^{+\infty}(e^{i\xi x}-1)\frac{dx}{x^\alpha}+i\xi\int_0^{+\infty}xe^{i\xi x}\frac{a_1(x)}{x^\alpha}dx\\
&+\int_0^{+\infty}(e^{i\xi x}-1)\frac{a_1(x)}{x^\alpha}dx
\end{align*}
Note in particular that the boundary terms disappear since $\alpha\in (1,2)$. Thus,
\begin{align*} 
\int_0^{+\infty}x(e^{i\xi x}-1)dF_X(x)&=c\alpha\xi^{\alpha-1}\int_0^{+\infty}(e^{i x}-1)\frac{dx}{x^{\alpha}}+i\xi\int_0^{+\infty}xe^{i\xi x}\frac{a_1(x)}{x^\alpha}dx\\
&+\int_0^{+\infty}(e^{i\xi x}-1)\frac{a_1(x)}{x^\alpha}dx
\end{align*}
To compute the first integral, we perform the following contour integration: we consider the function of the complex variable $f(z)=(e^{i z}-1)/z^\alpha$ and we integrate it along the contour formed by the upper quarter circle of radius $R$ minus the upper quarter circle of radius $0<r<R$. Letting $r\rightarrow 0$ and $R\rightarrow +\infty$, we obtain:
\begin{align*}
\int_0^{+\infty}(e^{ix}-1)\frac{dx}{x^{\alpha}}=ie^{-\frac{i\pi\alpha}{2}}\dfrac{\Gamma(2-\alpha)}{1-\alpha}.
\end{align*}
Thus, we have:
\begin{align*}
\int_0^{+\infty}x(e^{i\xi x}-1)dF_X(x)&=ic\alpha\xi^{\alpha-1}e^{-\frac{i\pi\alpha}{2}}\dfrac{\Gamma(2-\alpha)}{1-\alpha}+i\xi^{\alpha-1}\int_0^{+\infty}e^{ix}\frac{a_1\big(\frac{x}{\xi}\big)}{x^{\alpha-1}}dx\\
&+\xi^{\alpha-1}\int_0^{+\infty}(e^{i x}-1)\frac{a_1\big(\frac{x}{\xi}\big)}{x^\alpha}dx.
\end{align*}
Similarly, for the negative values of $X$, we have the following:
\begin{align*}
\int_{-\infty}^{0}x(e^{i\xi x}-1)dF_X(x)&=i\alpha c\xi^{\alpha-1}e^{+\frac{i\pi\alpha}{2}}\dfrac{\Gamma(2-\alpha)}{1-\alpha}+i\xi^{\alpha-1}\int_{-\infty}^0e^{ix}\frac{a_2\big(\frac{x}{\xi}\big)}{(-x)^{\alpha-1}}dx\\
&-\xi^{\alpha-1}\int_{-\infty}^0(e^{i x}-1)\frac{a_2\big(\frac{x}{\xi}\big)}{(-x)^\alpha}dx.
\end{align*}
Thus, using the explicit expression for $c$, we obtain:
\begin{align*}
\mathbb{E}\big[iX\exp\big(i\xi X\big)\big]&=-\alpha\xi^{\alpha-1}-\xi^{\alpha-1}\int_0^{+\infty}e^{ix}\frac{a_1\big(\frac{x}{\xi}\big)}{x^{\alpha-1}}dx\\
&+i\xi^{\alpha-1}\int_0^{+\infty}(e^{i x}-1)\frac{a_1\big(\frac{x}{\xi}\big)}{x^\alpha}dx-\xi^{\alpha-1}\int_{-\infty}^0e^{ix}\frac{a_2\big(\frac{x}{\xi}\big)}{(-x)^{\alpha-1}}dx\\
&-i\xi^{\alpha-1}\int_{-\infty}^0(e^{i x}-1)\frac{a_2\big(\frac{x}{\xi}\big)}{(-x)^\alpha}dx.
\end{align*}
Now, using the assumptions upon the functions $a_1(.)$ and $a_2(.)$ and Lebesgue dominated convergence theorem, we have:
\begin{align*}
\phi_X^*(0^+)=1.
\end{align*}
We proceed similarly for $\phi_X^*(0^-)$. The second bullet point is trivial. Let us prove the third and fourth ones. We consider the following linear functional on $S\big(\mathbb{R}\big)$:
\begin{align*}
F:\ \psi\longrightarrow \frac{1}{2\pi}\int_{\mathbb{R}}\mathcal{F}\big(\psi\big)(\xi)\phi_X^*(\xi)d\xi.
\end{align*}
By the previous computations, for any $\xi>0$ (and similarly for $\xi<0$), we have:
\begin{align*}
\phi_X^*(\xi)&=1+\frac{1}{\alpha}\int_0^{+\infty}e^{ix}\frac{a_1\big(\frac{x}{\xi}\big)}{x^{\alpha-1}}dx\\
&-\frac{i}{\alpha}\int_0^{+\infty}(e^{i x}-1)\frac{a_1\big(\frac{x}{\xi}\big)}{x^\alpha}dx+\frac{1}{\alpha}\int_{-\infty}^0e^{ix}\frac{a_2\big(\frac{x}{\xi}\big)}{(-x)^{\alpha-1}}dx\\
&+\frac{i}{\alpha}\int_{-\infty}^0(e^{i x}-1)\frac{a_2\big(\frac{x}{\xi}\big)}{(-x)^\alpha}dx.
\end{align*}
This implies the following simple bound on $\phi_X^*(\xi)$:
\begin{align*}
\mid\phi_X^*(\xi)\mid\leq 1+C_{\alpha}\big(\mid\mid a_1\mid\mid_{\infty}+\mid\mid a_2\mid\mid_{\infty}+\mid\mid x a_1(.)\mid\mid_{\infty}\mid\xi\mid+\mid\mid x a_2(.)\mid\mid_{\infty}\mid\xi\mid\big),
\end{align*}
for some constant, $C_{\alpha}$, strictly positive depending on $\alpha$ only. Thus, we have:
\begin{align*}
\mid F(\psi)\mid\leq C_{\alpha,X,p,q}\mid\mid\psi\mid\mid_{p,q,\infty},
\end{align*}
for some $p,q\in\mathbb{N}$ such that,
\begin{align*}
\mid\mid\psi\mid\mid_{p,q,\infty}=\underset{x\in\mathbb{R}}{\sup}\big((1+\mid x\mid)^p\mid\psi^{(q)}(x)\mid\big).
\end{align*}
This proves the third bullet point. For the fourth one, we introduce the following subspace of $S\big(\mathbb{R}\big)$:
\begin{align*}
S_0\big(\mathbb{R}\big)=\{\psi\in S\big(\mathbb{R}\big):\ \forall n\in\mathbb{N},\ \int_\mathbb{R}\psi(x)x^ndx=0 \}.
\end{align*}
Note that this space is invariant under the action of the pseudo-differential operator, $\mathcal{D}^{\alpha-1}$. Let $\psi\in S_0\big(\mathbb{R}\big)$. We have:
\begin{align*}
<T_X^*;\mathcal{D}^{\alpha-1}(\psi)>&=F\big(\mathcal{D}^{\alpha-1}(\psi)\big),\\
&=\frac{1}{2\pi}\int_{\mathbb{R}}\mathcal{F}\big(\mathcal{D}^{\alpha-1}(\psi)\big)(\xi)\phi_X^*(\xi)d\xi,\\
&=\frac{1}{2\pi}\int_{\mathbb{R}}\mathcal{F}\big(\psi\big)(\xi)\big(\dfrac{i\alpha\mid\xi\mid^\alpha}{\xi}\big)\dfrac{-\xi}{\alpha\mid\xi\mid^\alpha}(\phi_X(\xi))'d\xi,\\
&=\frac{i}{2\pi}\int_{\mathbb{R}}(\mathcal{F}\big(\psi\big)(\xi))'\phi_X(\xi)d\xi,\\
&=\mathbb{E}\big[X\psi(X)\big]
\end{align*}
To conclude, we have to prove the fifth bullet point. Let $\xi>0$. Note that the characteristic function of $Z^{\alpha}$ can not be equal to zero. Using straightforward computations, we have:
\begin{align*}
\dfrac{d}{d\xi}\bigg(\frac{\phi_X(\xi)- \phi_{Z^{\alpha}}(\xi)}{ \phi_{Z^{\alpha}}(\xi)}\bigg)\phi_{Z^{\alpha}}(\xi)&=\phi'_X(\xi)-\phi_X(\xi)\dfrac{\phi'_{Z^{\alpha}}(\xi)}{\phi_{Z^{\alpha}}(\xi)},\\
&=-\alpha\xi^{\alpha-1}\phi_X^*(\xi)+\phi_X(\xi)\alpha\xi^{\alpha-1},\\
&=\alpha\xi^{\alpha-1}\big(\phi_X(\xi)-\phi_X^*(\xi)\big).
\end{align*}
The result follows by integration.
\end{proof}
\noindent

\begin{proof}[Proof of Lemma \ref{orderN}]
Let $\xi>0$. By definition and using the fact that $X$ is centered, we have:
\begin{align*}
\phi_X\big(\xi\big)-1&=\int_{\mathbb{R}}\big(e^{ix\xi}-1\big)dF_X(x),\\
&=\int_{\mathbb{R}}\big(e^{ix\xi}-1-ix\xi\big)dF_X(x),\\
&=\int_{0}^{+\infty}\big(e^{ix\xi}-1-ix\xi\big)dF_X(x)+\int_{-\infty}^0\big(e^{ix\xi}-1-ix\xi\big)dF_X(x)
\end{align*}
Let us deal with the first term. By integration by parts, we have:
\begin{align*}
\int_{0}^{+\infty}\big(e^{ix\xi}-1-ix\xi\big)dF_X(x)=i\xi\int_{0}^{+\infty}\big(e^{ix\xi}-1\big)\dfrac{c+a_1(x)}{x^\alpha}dx
\end{align*}
Note that the boundary terms disappear since $\alpha\in (1,2)$. We have:
\begin{align*}
\int_{0}^{+\infty}\big(e^{ix\xi}-1-ix\xi\big)dF_X(x)&=i\xi\bigg[c\int_{0}^{+\infty}\big(e^{ix\xi}-1\big)\dfrac{dx}{x^\alpha}+\int_{0}^{+\infty}\big(e^{ix\xi}-1\big)\dfrac{a_1(x)}{x^\alpha}dx\bigg]\\
&=-c\xi^{\alpha}e^{-i\frac{\pi}{2}\alpha}\dfrac{\Gamma(2-\alpha)}{1-\alpha}+i\xi^{\alpha}\int_{0}^{+\infty}\big(e^{ix}-1\big)\dfrac{a_1\big(\frac{x}{\xi}\big)}{x^\alpha}dx.
\end{align*}
Similarly, for the term concerning negative values of $X$, we obtain the following expression:
\begin{align*}
\int_{-\infty}^0\big(e^{ix\xi}-1-ix\xi\big)dF_X(x)=-c\xi^{\alpha}e^{i\frac{\pi}{2}\alpha}\dfrac{\Gamma(2-\alpha)}{1-\alpha}-i\xi^\alpha\int_{-\infty}^{0}\big(e^{ix}-1\big)\dfrac{a_2\big(\frac{x}{\xi}\big)}{(-x)^\alpha}dx
\end{align*}
Thus, we have:
\begin{align*}
\phi_X\big(\xi\big)-1&=-\xi^{\alpha}+i\xi^{\alpha}\int_{0}^{+\infty}\big(e^{ix}-1\big)\dfrac{a_1\big(\frac{x}{\xi}\big)}{x^\alpha}dx\\
&-i\xi^\alpha\int_{-\infty}^{0}\big(e^{ix}-1\big)\dfrac{a_2\big(\frac{x}{\xi}\big)}{(-x)^\alpha}dx.
\end{align*}
This expression leads to the following easy bound:
\begin{align*}
\mid\phi_X\big(\xi\big)-1\mid\leq C_{\alpha}\mid\xi\mid^{\alpha}\big(1+\mid\mid a_1\mid\mid_{\infty}+\mid\mid a_2\mid\mid_{\infty}\big).
\end{align*}
A similar reasoning for $\xi<0$ ends the proof of the lemma.
\end{proof}

\begin{proof}[Proof of Lemma \ref{pointwise}] Let $R>0$ and $\xi>0$. We have:
\begin{align*}
\bigg\lvert\int_0^{+\infty}(e^{i x}-1)\frac{a_1\big(\frac{x}{\xi}\big)}{x^\alpha}dx\bigg\rvert \leq \bigg\lvert\int_0^{R}(e^{i x}-1)\frac{a_1\big(\frac{x}{\xi}\big)}{x^\alpha}dx\bigg\rvert+\bigg\lvert\int_R^{+\infty}(e^{i x}-1)\frac{a_1\big(\frac{x}{\xi}\big)}{x^\alpha}dx\bigg\rvert.
\end{align*}
Let us deal with the first term of the right-hand side of the previous inequality:
\begin{align*}
\bigg\lvert\int_0^{R}(e^{i x}-1)\frac{a_1\big(\frac{x}{\xi}\big)}{x^\alpha}dx\bigg\rvert &\leq \mid\mid a_1\mid\mid_\infty\int_0^{R}\frac{dx}{x^{\alpha-1}},\\
&\leq \mid\mid a_1\mid\mid_\infty \dfrac{R^{2-\alpha}}{2-\alpha}.
\end{align*}
For the second term, we have:
\begin{align*}
\bigg\lvert\int_R^{+\infty}(e^{i x}-1)\frac{a_1\big(\frac{x}{\xi}\big)}{x^\alpha}dx\bigg\rvert&\leq 2\int_R^{+\infty}\dfrac{\mid a_1\big(\frac{x}{\xi}\big)\mid}{x^\alpha}dx,\\
&\leq 2  \mid\mid xa_1(.) \mid\mid_{\infty}\mid\xi\mid\int_R^{+\infty}\dfrac{dx}{x^{\alpha+1}},\\
&\leq 2  \mid\mid xa_1(.) \mid\mid_{\infty}\mid\xi\mid\dfrac{R^{-\alpha}}{\alpha}.
\end{align*}
Optimising in $R>0$ leads to:
\begin{align*}
\bigg\lvert\int_0^{+\infty}(e^{i x}-1)\frac{a_1\big(\frac{x}{\xi}\big)}{x^\alpha}dx\bigg\rvert \leq C^1_{\alpha,1}\mid \xi\mid^{\frac{2-\alpha}{2}},
\end{align*}
for some $C^1_{\alpha,1}>0$ depending on $\mid\mid xa_1(.) \mid\mid_{\infty}$, $\mid\mid a_1 \mid\mid_{\infty}$ and $\alpha$ only. We proceed exactly in the same way to obtain the appropriate estimates for the other integrals.
\end{proof}



\end{document}